\documentclass[a4paper,reqno]{amsart}
\usepackage{amsthm, amssymb, amsmath, latexsym}
\usepackage{amsfonts}
\usepackage{graphicx}
\usepackage{color}
\usepackage[colorinlistoftodos]{todonotes}

\usepackage{mathabx}
\usepackage{breqn}

\usepackage{esint}

\theoremstyle{plain}
\begingroup
\newtheorem{theorem}{Theorem}[section] 
\newtheorem{lemma}[theorem]{Lemma}
\newtheorem{proposition}[theorem]{Proposition}

\endgroup

\theoremstyle{definition}
\begingroup
\newtheorem{definition}[theorem]{Definition}
\newtheorem{remark}[theorem]{Remark}

\endgroup

\theoremstyle{remark}
\begingroup
\endgroup

\mathchardef\emptyset="001F

\numberwithin{equation}{section}

\newcommand{\Om}{\Omega}
\newcommand{\R}{\mathbb{R}}
\newcommand{\Rd}{\mathbb{R}^2}
\newcommand{\Rdd}{\mathbb{R}^{2\times 2}}

\newcommand{\Ls}{{L^2}}

\newcommand{\V}{\mathbb V}
\newcommand{\C}{\mathbb C}

\newcommand{\A}{\mathbb A}

\newcommand{\dt}{\textnormal{d}t}

\newcommand{\dtau}{\textnormal{d}\tau}

\newcommand{\dive}{\textnormal{div}}

\newcommand\ee{\end{equation}}
\newcommand\be{\begin{equation}}

\title[Dynamic crack growth in viscoelastic materials with memory]
{Dynamic crack growth in\\ viscoelastic materials with memory}

\author[Federico Cianci]{Federico Cianci}
\address[Federico Cianci]{SISSA, Via Bonomea 265, 34136 Trieste,
Italy}
\email[Federico Cianci]{fcianci@sissa.it}

\begin{document}

\begin{abstract}

In this paper we introduce a model of dynamic crack growth in viscoelastic materials, where the damping term depends on the history of the deformation. The model is based on a dynamic energy dissipation balance and on a maximal dissipation condition. Our main result is an existence theorem in dimension two under some a priori regularity constraints on the cracks.
\end{abstract}

\maketitle


{\small{{\it 2020 MSC:\/}  
35Q74
, 74D05
, 74H20
.

}}

\section{Introduction}\label{intro}

We consider the problem of crack growth in a viscoelastic material with memory governed by the system
\begin{equation}\label{intro-mainprobl-nc-infty}
    \ddot{u}(t) - \dive\big{(}(\C+\V) Eu(t)\big{)} + \dive\Big{(} \int^t_{-\infty} \textnormal{e}^{\tau-t}\,\V Eu(\tau) \, \textnormal{d}\tau \Big{)} = f(t),
\end{equation}
where $u$, $Eu$, and $\ddot{u}$, are the displacement, the symmetric part of its gradient, and its second derivative with respect to time, $\C$ and $\V$ are the elasticity and viscosity tensors, while $f$ is the external load. For this model the stress at time $t$ is given by
\begin{equation}\label{stress-orig}
    \sigma(t):=\C Eu(t) + \V Eu(t) - \int^t_{-\infty} \textnormal{e}^{\tau-t}\,\V Eu(\tau) \, \textnormal{d}\tau.
\end{equation}
Moreover, as in \cite{Dafermos, Fab-Gi-Pata} we assume that we know the displacement $u$ on $(-\infty,0]$ and we want to solve \eqref{intro-mainprobl-nc-infty} on $[0,T]$, for given $T>0$. It is convenient to write \eqref{intro-mainprobl-nc-infty} in the form
\begin{equation}\label{intro-mainprobl-nc-l0}
    \ddot{u}(t) - \dive(\sigma_0(t))= \ell_0(t) \quad t\in [0,T],
\end{equation}
where
\begin{equation}\label{tensore-da-0}
    \sigma_0(t):=\C Eu(t)+\V Eu(t) - \int^t_{0} \textnormal{e}^{\tau-t}\,\V Eu(\tau) \, \textnormal{d}\tau,
\end{equation}
\begin{equation}
    \ell_0(t):=f(t) - \dive F_0(t),
\end{equation}
\begin{equation}\label{F_0}
     F_0(t):=\int^0_{-\infty} \textnormal{e}^{\tau-t}\,\V Eu_0(\tau) \, \textnormal{d}\tau
\end{equation}
and $u_0$ is a function that represents the displacement on $(-\infty,0]$, namely $u(s)=u_0(s)$ for every $s\in (-\infty,0]$.

When no cracks are present, problems similar to \eqref{intro-mainprobl-nc-infty} and \eqref{intro-mainprobl-nc-l0} were studied by Boltzmann (\cite{Boltz_1}, \cite{Boltz_2}) and Volterra (\cite{Volterra_1}, \cite{Volterra_2}), while recent results can be found in \cite{DL_V1}, \cite{Fab-Gi-Pata}, \cite{Fab-Morro}, and \cite{Slepyan}.

In this paper we study the problem on a bounded open set $\Omega\subset \R^2$. The crack at time $t\in[0,T]$ is a 1-dimensional closed subset $\Gamma_t$ of $\Omega$ and the irreversibility of crack growth means that $\Gamma_t \subseteq \Gamma_\tau$ if $t \leq \tau$. For technical reasons we assume that the shape of the cracks and their dependence on time is sufficiently regular, with precise a priori estimates.

In the case of smooth functions, equation \eqref{intro-mainprobl-nc-l0} is satisfied on $\Omega\setminus \Gamma_t$ with suitable boundary conditions (on the Dirichlet part $\partial_D\Omega$ and on the Neumann $\partial_N\Omega$ of $\partial\Omega$) and with prescribed initial conditions. Namely, $u$ and $\{\Gamma_t\}_{t\in [0,T]}$ satisfy
\begin{alignat}{2}
    & \ddot{u}(t) - \dive (\sigma_0(t)) = \ell_0(t) \qquad &&\textnormal{in } \Om\setminus\Gamma_t,\label{eq:forte1}\\
    &u(t) = u_D(t) \qquad && \text{on $\partial_D\Om$,} \label{eq:forte2} \\
    &\sigma_0(t)\nu=F_0(t)\nu\qquad && \text{on $\partial_N\Om$,}\\
    &\sigma^{\pm}_0(t)\nu=F^\pm_0(t)\nu\qquad && \text{on $\Gamma_t$,}\label{eq:bcNCrack}\\
    &u(0)=u^0\quad\textnormal{and}\quad\dot u(0)=u^1\quad && {}\label{eq:forte5} 
    \end{alignat}
    for every $t\in[0,T]$, where $u_D$ is the Dirichlet condition, $u^0$ is the initial condition for the displacement, $u^1$ is the initial condition for the velocity, $\nu$ is the unit normal, and the symbol $\pm$ in \eqref{eq:bcNCrack} denotes suitable limits on each side of $\Gamma_t$. In the paper we consider a weak formulation (see Definition \ref{def:weak-sol-t0t1}) which coincides with the one in \eqref{eq:forte1}-\eqref{eq:forte5} under suitable regularity assumptions.

When $\{\Gamma_t\}_{t\in [0,T]}$ is prescribed, problem \eqref{eq:forte1}-\eqref{eq:forte5} has been studied in \cite{Sapio} and \cite{Cianci-Dalmaso}. More precisely, in \cite{Sapio} an existence theorem is proved, while in \cite{Cianci-Dalmaso} one can find results regarding uniqueness and continuous dependence of $u$ on the data (in particular on the cracks).

In the model considered in our paper the unknown of the problem is the family of cracks $\{\Gamma_t\}_{t\in [0,T]}$ which, in the spirit of \cite{Dalmaso-Larsen-Toader2016} and \cite{Dalmaso-Larsen-Toader}, must satisfies the following conditions:
\begin{itemize}
    \item[a)] an energy dissipation balance (consistent with dynamic Griffith's theory) for the solution $u$ of \eqref{eq:forte1}-\eqref{eq:forte5} (see Definition \ref{def:Creg-t0t1}): the sum of the kinetic and elastic energies and of the energies dissipated by viscosity and crack growth balances the work done by the forces acting on the system;
    \item[b)] a maximal dissipation condition, depending on a parameter $\eta>0$ (see Definition \ref{def:maximal-dissip}), which forces the crack to run as fast as possible.
\end{itemize}
Condition a) is a dynamic version of Griffith's criterion (see \cite{Griffith} for the quasistatic case and \cite{Mott} for the dynamic problem).

The main result of this paper is that, given initial and boundary conditions satisfying suitable hypotheses, there exists a $\{\Gamma_t\}_{t\in [0,T]}$ satisfying a) and b) (see Theorem \ref{thm:main-dissipat}).

The proof follows the lines of \cite{Dalmaso-Larsen-Toader}, where a similar problem is studied for the case of pure elastodynamics. To deal with the memory term appearing in \eqref{tensore-da-0}, we use the results of \cite{Sapio} and \cite{Cianci-Dalmaso}. In particular the continuous dependence on the data obtained in \cite{Cianci-Dalmaso} is a fundamental tool for a compactness argument that plays a key role in the proof of Theorem \ref{thm:main-dissipat}.

The structure of the paper is the following:
    \begin{itemize}
        \item in Section \ref{Formulation} we give a precise formulation of the problem and we give all the preliminary results;
        
        \item in Section \ref{sect:energy-balance} we define the class of cracks $\{\Gamma_t\}_{t\in [0,T]}$ such that the energy balance described in a) is satisfied and we prove a compactness result;
        
        \item in Section \ref{sect:existence-coupled} we define the maximal dissipation condition and we prove the main result of the paper (Theorem \ref{thm:main-dissipat}).
    \end{itemize}

\section{Formulation of the problem}\label{Formulation}

The reference configuration of our problem
is a bounded open set $\Om\subset\mathbb{R}^2$, with Lipschitz boundary $\partial\Om$ and we assume that $\partial\Om=\partial_D\Om\cup \partial_N\Om$, where  $\partial_D\Om$ and $\partial_N\Om$ are disjoint (possibly empty) Borel sets, on which we prescribe
Dirichlet and Neumann boundary conditions respectively. Moreover, we fix a time interval $[0,T]$, with $T>0$.

We give a precise definition of the admissible cracks of our model using a suitable class of curves. The following definitions and results are based on \cite{Dalmaso-Larsen-Toader2016} and \cite{Dalmaso-Larsen-Toader}.
The curves are always parameterized using the arc-length parameter $s$ and for a given curve $\gamma: [a_\gamma, b_\gamma] \to \mathbb{R}^2$ we define $\Gamma^\gamma:= \gamma([a_\gamma,\,b_\gamma])$ and $\Gamma^\gamma_s:= \gamma([a_\gamma,\,s])$, for every $s \in [a_\gamma,\,b_\gamma]$. When it is clear from the context we omit the dependence on $\gamma$ and we write $\Gamma$ and $\Gamma_s$ instead of $\Gamma^\gamma$ and $\Gamma^\gamma_s$. In order to describe the initial crack, we fix a curve $\gamma_0: [a_0,\, 0] \to \overline{\Omega}$ such that $\gamma_0(a_0)\in \partial\Omega$, $\gamma_0(s)\in\Omega$ for every $s\in (a_0,0]$ and we define the initial crack as
\begin{equation*}
    \Gamma_0:= \gamma_0([a_0, 0]).
\end{equation*}
We suppose that $\gamma_0$ is of class $C^{3,1}$ and that it is transversal to $\partial \Omega$ at $\gamma_0(a_0)$ (there exists an isosceles triangle
contained in $\overline{\Omega}$ with vertex in $\gamma_0(a_0)$ and axis parallel to $\gamma'_0(a_0)$). We fix two constants $r>0$ and $L>0$ and we now define the space of admissible crack paths.

\begin{definition}\label{def:spazioGrL}
Let $\mathcal{G}_{r,L}$ be the space of simple curves $\gamma:[a_0, b_\gamma] \to \overline{\Omega}$ of class $C^{3,1}$, with $a_0 < 0 \leq b_\gamma$, such that
\begin{itemize}
    \item[(a)] $\gamma(s)=\gamma_0(s)$ for every $s \in [a_0,\,0]$,
    \item[(b)] $|\gamma'(s)|=1$ for every $s \in [a_0, b_\gamma]$,
    \item[(c)] the two open disks of radius $r$ tangent to $\Gamma$ at $\gamma(s)$ do not intersect $\Gamma$,
    \item[(d)] $\textnormal{dist}(\gamma([0,b_\gamma]),\partial\Omega) \geq 2r $,
    \item[(e)] $|\gamma^{(3)}(s)|\leq L$, $|\gamma^{(3)}(s_2)-\gamma^{(3)}(s_1)|\leq L |s_2-s_1|$ for any $s,\,s_1,\,s_2\in [a_0,b_\gamma]$,
\end{itemize}
where $\gamma^{(i)}$ denotes the $i-$th derivative of $\gamma$.
\end{definition}

We fix $\gamma_0$, $r$, and $L$ such that $\mathcal{G}_{r,L}\neq \emptyset$.

\begin{remark}
 By (a) and (d) we have $|a_0| \geq 2r$. Condition (c) implies $|\gamma^{(2)}(s)| \leq 1/r$ for every $s \in [a_0,\,b_\gamma]$.
\end{remark}

\begin{definition}\label{def:conv-unif}
Let $\gamma_k$ be a sequence of curves in $\mathcal{G}_{r,L}$ and let $\gamma \in \mathcal{G}_{r,L}$. We say that $\gamma_k$ converges uniformly to $\gamma$ if $b_{\gamma_k} \to b_\gamma$ and for every $b\in (0,b_\gamma)$ we have $\gamma_k|_{[a_0,b]} \to \gamma|_{[a_0,b]}$ uniformly in $[a_0,b]$.
\end{definition}

\begin{lemma}\label{lemma:estensione}
 There exist two constants $\hat r$ and $\hat L$, with $0 < \hat r < r$ and $\hat L> L$, depending only on $r$ and $L$, such that for every $\gamma \colon [a_0, b_\gamma] \to \overline\Omega$ with $\gamma \in \mathcal{G}_{ r,  L}$ there exists an extension $\hat\gamma \colon [a_0, b_\gamma + \hat r] \to \overline\Omega$ of $\gamma$ with $\hat\gamma \in \mathcal{G}_{\hat r, \hat L}$, whose image will be indicated by $\hat \Gamma$. Moreover, the extension can be chosen in such a way that the uniform convergence of $\gamma_k$ implies the uniform convergence of the corresponding extensions $\hat \gamma_k$.
\end{lemma}

\begin{lemma}\label{lemma:compattezza-gamma}
 Let $\gamma_k$ be a sequence of curves in $\mathcal{G}_{r,L}$. Then there
exist a subsequence, not relabelled, and a curve $\gamma\in \mathcal{G}_{r,L}$ such that $\gamma_k$ converges to $\gamma$ uniformly.
\end{lemma}

For a proof of the previous two lemmas see \cite{Dalmaso-Larsen-Toader}.

We have to describe the dependence of the crack length on the time. We fix two constants $\mu>0$ and $M>0$ which bound the speed of the crack tip and some higher order derivatives of the crack length, respectively.

\begin{definition}\label{def:lunghezze}
Let $T_0 < T_1$. The class $\mathcal{S}^{reg}_{\mu,M}(T_0,T_1)$ is composed of all nonnegative functions satisfying the following conditions:
\begin{equation}
    s \in C^{3,1}([T_0,T_1]),
\end{equation}
\begin{equation}
    0 \leq \dot{s}(t) \leq \mu
\end{equation}
\begin{equation}
    |\ddot s(t)| \leq M,\,\,\, |\dddot s(t)| \leq M, \,\,\, |\dddot s(t_1) -\dddot s(t_2)| \leq M |t_1 - t_2|,
\end{equation}
for $t, t_1, t_2 \in [T_0,T_1]$, where dots denote derivatives with respect to time. We denote by $\mathcal{S}^{piec}_{\mu,M}(T_0,T_1)$ the set of all functions $s\in C^0([T_0,T_1])$ such that there exists a finite subdivision $T_0=\tau_0 < \tau_1 < ... < \tau_k = T_1$ for which $s|_{[\tau_{j-1},\tau_j]} \in \mathcal{S}^{reg}_{\mu,M}(\tau_{j-1},\tau_j)$. The minimal set $\{\tau_0, \tau_1, ..., \tau_k \}$ for which this property holds is denoted by $sing(s)$.
\end{definition}

Given $0\leq T_0 < T_1 \leq T$, $\gamma \in \mathcal{G}_{r,L}$, $s\in \mathcal{S}^{piec}_{\mu,M}(T_0,T_1)$, with $s(T_1) \leq b_\gamma$, the time dependent cracks corresponding to these functions are given by
\begin{equation*}
    \Gamma^\gamma_{s(t)}:=\gamma([a_0,s(t)]) \quad \textnormal{for all } t \in [T_0,T_1],
\end{equation*}
and the corresponding cracked domains are
\begin{equation*}
    \Omega^\gamma_{s(t)}:=\Omega\setminus\Gamma^\gamma_{s(t)} \quad \textnormal{for all } t \in [T_0,T_1].
\end{equation*}
For simplicity of notation we sometimes denote $\Gamma^\gamma_{s(t)}$ by $\Gamma_{s(t)}$, when $\gamma$ is clear from the context.

In \cite{Caponi}, \cite{Caponi-tesi}, and \cite{Cianci-Dalmaso} the cracks are described using a family of time-dependent diffeomorphism $\Phi, \Psi: [0, T] \times \overline{\Om} \to \overline{\Om}$. Thanks to the following result it is possible to obtain the same maps also in our case. For a proof see \cite[Lemma 2.8]{Dalmaso-Larsen-Toader2016}.

\begin{lemma}\label{lemma:costruzione-diffeom}
Let $\varepsilon>0$ and let $\rho\in (0,\hat r /2)$, where $\hat r$ is the constant that appears in Lemma \ref{lemma:estensione}. Then there exists two constants $\delta \in (0,\rho/\mu)$ and $C>0$ depending only on $r$, $L$, $\mu$, $M$, $\varepsilon$, and $\rho$, with the following property: for every $\gamma\in \mathcal{G}_{r,L}$, for every $t_0<t_1$, and for every $s \in \mathcal{S}^{reg}_{\mu,M}(t_0, t_1)$, with $t_1-t_0 \leq \delta$, $s(t_1)\leq b_\gamma$, we can define two functions $\Phi,\Psi\colon [t_0,t_1] \times \overline{\Omega} \to \overline{\Omega}$ of class $C^{2,1}$ with the following properties:
\begin{itemize}
    \item[(a)] for every $t\in [t_0,t_1]$ we have $\Phi(t,\overline{\Omega})=\overline{\Omega}$, $\Phi(t,\hat \Gamma)=\hat \Gamma$ (where $\hat\Gamma$ is the set that appears in Lemma \ref{lemma:estensione}), $\Phi(t,\Gamma_{s(t_0)})=\Gamma_{s(t)},$ and $\Phi(t,y)=y$ on $\overline{\Omega}\setminus B(\gamma(s(t_0)),2\rho);$
    \item[(b)] $\Phi(t_0,y)=y$ for every $y \in \overline{\Omega}$;
    \item[(c)] for every $t\in [t_0,t_1]$, $\Psi(t,\cdot)$ is the inverse of $\Phi(t,\cdot)$ on $\overline{\Omega}$;
    \item[(d)] for every $t\in [t_0,t_1]$ we have $1-\varepsilon \leq \det D\Phi(t,y) \leq 1 + \varepsilon$ and $1-\varepsilon \leq \det D\Psi(t,y) \leq 1 + \varepsilon$ for every $x,y\in\overline\Omega$, where $D$ denotes the spatial jacobian matrix.
    \item[(e)] for every $t\in [t_0,t_1]$ we have $|\partial_t\Phi(t,y)|\leq \mu(1+\varepsilon)$ for every $y\in \overline\Omega$;
    \item[(f)] the absolute values of all partial derivatives of $\Phi$ and of $\Psi$ of order less than or equal to two, as well as the Lipschitz constants of all second derivatives, are bounded by $C$;
    \item[(g)] if $\gamma_k\in \mathcal{G}_{r,L}$ converges to $\gamma$ uniformly, if $s_k\in \mathcal{S}^{reg}_{\mu,M}(t_0,t_1)$ converges to $s$ uniformly, with $s_k(t_1)\leq b_{\gamma_k}$ for every $k$, then the corresponding diffemorphisms satisfy $\Phi_k(t,x)\to \Phi(t,x)$ for every $t\in [t_0,t_1]$ and for every $x\in \overline{\Omega}.$
\end{itemize}
\end{lemma}

We now define the functional spaces that will be used in order to give the definition of weak solution of the viscoelastic problem \eqref{eq:forte1}-\eqref{eq:forte5}.

We define $\R^{2\times 2}$ as the space of real $2\times 2$ matrix and $\R^{2\times 2}_{sym}$ as the space of real $2\times 2$ symmetric matrices. The euclidean scalar product between the matrices $A$ and $B$ is denoted by $A:B$. For every $A\in \mathbb{R}^{2\times 2}$ the symmetric part $A_{sym}\in \mathbb{R}^{2\times2}$ is defined as $A_{sym}=\frac{1}{2}(A+A^T)$, where $A^T$ denotes the transpose matrix of $A$.  For any pair of vector spaces we define $\mathcal{L}(X; Y)$ as the space of linear and continuous maps form $X$ into $Y$. Let $0 < \lambda < \Lambda$ be two fixed constants. We now define the space of tensors that will be used in the paper.

\begin{definition}\label{def:spazio-tensori}
We define $\mathcal{E}(\lambda,\Lambda)$ as the set of all maps $\mathbb{L}:\overline\Omega \to \mathcal{L}(\R^{2\times 2}; \R^{2\times 2} )$ of class $C^2$ such that for every $x\in \overline{\Omega}$ we have 
\begin{equation}
    \mathbb{L}(x)A=\mathbb{L}(x)A_{sym}\in\mathbb{R}^{2\times 2}_{sym} \quad \textnormal{for every } A \in \mathbb{R}^{2\times 2},
\end{equation}
\begin{equation}
    \mathbb{L}(x)A:B=\mathbb{L}(x)B:A\quad \textnormal{for every } A,\,B \in \mathbb{R}^{2\times 2},
\end{equation}
\begin{equation}
    \lambda|A_{sym}|^2\leq \mathbb{L}(x)A:A \leq \Lambda|A_{sym}|^2 \quad \textnormal{for every } A \in \mathbb{R}^{2\times 2}.
\end{equation}
\end{definition}

We now fix the following maps 
\begin{equation}\label{eq:C-V}
\C,\V \in \mathcal{E}(\lambda,\Lambda), \quad \A:=\C+\V
\end{equation}
where $\C(x)$ and $\V(x)$ respectively represent the elasticity and viscosity tensor at the point $x\in \overline\Omega$.

Given $\gamma\in \mathcal{G}_{r,L}$, $0\leq T_0 < T_1 \leq T$, and $s \in \mathcal{S}^{piec}_{\mu,M}(T_0,T_1)$, with $s(T_1)\leq b_\gamma$, we now introduce the function spaces that will be used in the precise formulation of problem \eqref{eq:forte1}-\eqref{eq:forte5}.

We recall that $\Gamma:=\gamma([a_0,b_\gamma]; \mathbb{R}^2)$. For every $u\in H^1(\Omega\setminus\Gamma;\mathbb{R}^2)$ $Du$ denotes jacobian matrix in the sense of distributions on $\Omega\setminus\Gamma$ and $Eu$ is its symmetric part, i.e.,
\begin{equation*}
    Eu:=\tfrac{1}{2}(Du + Du^T).
\end{equation*}

The following lemma is an extension of the second Korn's inequality (see, e.g., \cite{Ol-Sha-Yos}) to the case of cracked domain. For a proof see, e.g., \cite{Cianci-Dalmaso}.
\begin{lemma}
 Let $\gamma \in \mathcal{G}_{r,L}$ and let $\Gamma:=\gamma([a_0,b_\gamma]; \mathbb{R}^2)$. Then there exists a constant $K$, depending only on $\Om$ and $\Gamma$, such that
\begin{equation}
\|Du\|^2 \leq K ( \|u\|^2 + \|Eu\|^2 ) 
\end{equation}
for every $u\in H^1(\Omega\setminus\Gamma;\mathbb{R}^2)$, where $\|\cdot\|$ denotes the $L^2$ norm.
\end{lemma}

\begin{remark}
  Let $\gamma \in \mathcal{G}_{r,L}$ and let $\Gamma:=\gamma([a_0,b_\gamma]; \mathbb{R}^2)$. Then, using a localization argument (see, e.g., \cite{Cianci-Dalmaso}), we can prove that the trace operator is well defined and continuous from $H^1(\Omega\setminus\Gamma;\mathbb{R}^2)$ into $L^2(\partial\Omega;\mathbb{R}^2)$.
\end{remark}

We set
\begin{equation}\label{eq:VeH-t0t1}
    V^\gamma:=H^1(\Omega\setminus\Gamma;\Rd), \quad H:=L^2(\Omega;\Rd{}), \quad\textnormal{and} \quad \underline{H}:=L^2(\Omega;\Rdd{})
\end{equation}
Since $\mathcal{L}^2(\Gamma)=0$, we have the embedding $V^\gamma \hookrightarrow H \times \underline{H}$ given by $v \mapsto (v,\,Dv) $ and we can see the distrubutional gradient $Dv$ on $\Omega\setminus \Gamma$ as a function defined a.e. on $\Omega$, which belongs to $\underline{H}$.
 
For every finite dimensional Hilbert space $Y$ the symbols $(\cdot\,,\cdot)$ and $\|\cdot\|$ denote the scalar product and the norm in the $L^2(\Omega; Y)$, according to the context.
The space $V^\gamma$ is endowed with the norm
\begin{equation}
    \|u\|_{V^\gamma}:=\big{(}\|u\|^2+ \|Du\|^2\big{)}^{{1}/{2}}.
\end{equation}
For every $\overline s\in [a_0,b_\gamma]$ we define \begin{equation}\label{eq:def-spazi-t0t1}
V^{\gamma}_{\overline{s}}:=H^1(\Omega\setminus\Gamma_{\overline{s}};\mathbb{R}^2) \quad \textnormal{and} \quad V^{\gamma,D}_{\overline{s}}:=\big\{u\in V^{\gamma}_{\overline{s}}\,\,\big|\,\,u|_{\partial_D\Om}=0\big\},
\end{equation}
where $\Gamma_{\overline{s}}=\gamma([a_0,{\overline{s}}])$ and $u|_{\partial_D\Om}$ denotes the trace of $u$ on $\partial_D\Omega$. We note that $V^{\gamma}_{\overline{s}}$ and $V^{\gamma,D}_{\overline{s}}$ are closed linear subspaces of $V^{\gamma}$. For every $t\in [T_0,T_1]$ the spaces $V^{\gamma}_{{s}(t)}$ and $V^{\gamma,D}_{{s}(t)}$ are defined as in \eqref{eq:def-spazi-t0t1} with $\overline{s}=s(t)$.

We define
\begin{equation}\label{eq:Vcorsivo-t0t1}
    \mathcal{V}_{\gamma,s}(T_0,T_1):=\big\{ v\in \Ls(T_0,T_1; V^\gamma)\cap H^1(T_0,T_1;H)\,\big|\,v(t)\in V^{\gamma}_{s(t)}\!\text{  for a.e. }\!t\!\in\! (T_0,T_1) \big\},
\end{equation}
which is a Hilbert space with the norm
\begin{equation}\label{eq:norma_U}
    \|v\|_{\mathcal{V}_{\gamma,s}}:=\big{(}\|v\|^2_{\Ls(T_0,T_1;V^\gamma)}+\|\dot v\|^2_{\Ls(T_0,T_1;H)}\big{)}^{\frac{1}{2}},
\end{equation}
where the dot denotes the distibutional derivative with respect to $t$. Moreover we set
\begin{equation}\label{eq:VcorsivoD-t0t1}
     \mathcal{V}^D_{\gamma,s}(T_0,T_1):=\big\{ v\in \mathcal{V}_{\gamma,s}(T_0,T_1) \,\big|\,v(t)\in V^D_{s(t)}\,\text{  for a.e. }\,t\in (T_0,T_1) \big\},
\end{equation}
which is a closed linear subspace of $\mathcal{V}_{\gamma,s}(T_0,T_1)$ and we define
\begin{equation}\label{eq:Vcorsivo-infty-t0t1}
    \!\mathcal{V}^\infty_{\gamma,s}(T_0,T_1)\!:=\!\big\{ v\!\in\! L^\infty(T_0,T_1; V^\gamma)\cap W^{1,\infty}(T_0,T_1;H)\,\big|\,v(t)\!\in\! V^{\gamma}_{s(t)}\!\text{  for a.e. }\!t\!\in\! (T_0,T_1) \big\},
\end{equation}
which is a Banach space with the norm
\begin{equation}\label{eq:norma_Uinfty}
    \|v\|_{\mathcal{V}^\infty_{\gamma,s}}:=\|v\|_{L^\infty(T_0,T_1;V^\gamma)}+\|\dot v\|_{L^\infty(T_0,T_1;H)}.
\end{equation}
Moreover, it is convenient to introduce the space of weakly continuous functions with values in a Banach space $X$ with topological dual $X^*$, defined by
\begin{equation*}
    C^0_w([T_0,T_1]; X):=\big\{ v: [T_0,T_1] \to X \,\big|\, t \mapsto \langle h,\, v(t) \rangle \textnormal{ is continuous for every } h \in X^* \big\}.
\end{equation*}

When it is clear from the context we will omit the dependence on $\gamma$ or $s$ in the functional spaces, writing $V$, $V_{s(t)}$, $V^{D}_{s(t)}$, $\mathcal{V}(T_0,T_1)$, $\mathcal{V}^D(T_0,T_1)$, and $\mathcal{V}^\infty(T_0,T_1)$ instead of $V^\gamma$, $V^{\gamma}_{s(t)}$, $V^{\gamma,D}_{s(t)}$, $\mathcal{V}_{\gamma,s}(T_0,T_1)$, $\mathcal{V}^D_{\gamma,s}(T_0,T_1)$, and $\mathcal{V}^\infty_{\gamma,s}(T_0,T_1)$.

Since $H^1(T_0,T_1;H)\hookrightarrow C^0([T_0,T_1];H)$ we have $\mathcal{V}(T_0,T_1)\hookrightarrow C^0([T_0,T_1],H)$. In particular $v(T_0)$ and $v(T_1)$ are well defined elements of $H$, for every $v\in \mathcal{V}(T_0,T_1)$.

We set
\begin{equation}
    \tilde{H}:=L^2(\Omega;\R^{2\times 2}_{sym}).
\end{equation}
On the forcing term $\ell(t)$ of \eqref{eq:forte1} we assume that
\begin{equation}
    \ell(t):= f(t) - \dive F(t),
\end{equation}
where
\begin{equation}\label{eq:forzante-f-t0t1}
    f\in \Ls(0,T;H) \quad \text{and} \quad
    F\in H^1(0,T;\tilde{H})
\end{equation}
are prescribed functions and  the divergence of a matrix valued function is the vector valued function whose components are obtained taking the divergence of the rows.

The Dirichlet boundary condition on $\partial_D \Omega$ is obtained by prescribing a function
\begin{equation}\label{eq:dato-u_D}
    u_D\in H^2(0,T;\,H)\cap H^1(0,T;\,V_{0}).
\end{equation}
where $V_0$ is $V_{\overline{s}}$ for $\overline{s}=0$. It is not restrictive to assume that for every $t\in [0,T]$
\begin{equation}\label{eq:condizione-nulla-dentro}
    u_D(t)=0 \quad \textnormal{a.e. on } \{ x\in \Omega \,|\, \textnormal{dist}(x,\partial\Omega) \geq r \}.
\end{equation}
Moreover we will prescribe the natural Neumann boundary condition on $\partial_N\Omega \cup \Gamma_t$.

We are now in a position to give the definition of weak solution for the viscoelastic problem.

\begin{definition}[Solution for visco-elastodynamics with cracks]\label{def:weak-sol-t0t1}
Let $\gamma\in \mathcal{G}_{r,L}$, $0\leq T_0 < T_1 \leq T$, $s \in \mathcal{S}^{piec}_{\mu,M}(T_0,T_1)$, with $s(T_1)\leq b_\gamma$, and assume \eqref{eq:C-V}, \eqref{eq:forzante-f-t0t1}-\eqref{eq:condizione-nulla-dentro}.  Let
$ u^{0} \in V_{s(T_0)}$, such that $u^{0}-u_D(T_0)\in V^D_{s(T_0)}$ and let $u^{1}\in H$. We say that $u$ is a weak solution of the problem of visco-elastodynamics on the cracked domains $\Omega\setminus\Gamma_{s(t)}$, $t\in[T_0,T_1]$, with initial conditions $u^{0}$ and $u^{1}$, if
\begin{align}
    &\label{eq:weak_generalizzata1-t0t1} \,\,\,u\in\mathcal{V}(T_0,T_1) \quad \text{and} \quad u-u_D\in \mathcal{V}^D(T_0,T_1),\\
    &-\int^{T_1}_{T_0}(\dot u(t),\dot \varphi(t))\,\dt+ \int^{T_1}_{T_0} (\A Eu(t),E\varphi(t)) \,\dt \vspace{1cm}\nonumber \\
    & - \int^{T_1}_{T_0}\int^t_{T_0}  \textnormal{e}^{\tau-t} (\V Eu(\tau), E\varphi(t))\,\dtau\dt = \int^{T_1}_{T_0} (f(t),\varphi(t))\,\dt{} \nonumber\\
    &\label{eq:weak_generalizzata2-t0t1} + \int^{T_1}_{T_0}  (F(t),E\varphi(t))  \,\dt{} \,\,\,\text{ for all }  \varphi\in \mathcal{V}^D(T_0,T_1) \text{ with }\varphi(T_0)=\varphi(T_1)=0,\\
    & \label{eq:weak_generalizzata3-t0t1}\,\,\,u(T_0)=u^{0} \quad \text{in }  H \quad  \text{and} \quad \dot u(T_1)=u^{1} \quad \text{in } (V^D_{s(T_0)})^*,
\end{align}
where $(V^D_{s(T_0)})^*$ denotes the topological dual of $V^D_{s(T_0)}$.
\end{definition}
\begin{remark}
 If $u$ satisfy \eqref{eq:weak_generalizzata1-t0t1} and \eqref{eq:weak_generalizzata2-t0t1}, it is possible to prove that $\dot u\in H^1(0,T;(V^D_{s(T_0)})^*)$ (see \cite[Remark 4.6]{Sapio}), which implies $\dot u\in C^0([T_0,T_1];(V^D_{s(T_0)})^*)$. In particular $\dot u(T_0)$ is well defined as an element of $(V^D_{s(T_0)})^*$.
\end{remark}

\begin{remark}
 In the case of smooth functions problem \eqref{eq:weak_generalizzata1-t0t1}-\eqref{eq:weak_generalizzata3-t0t1} is satisfied in a stronger sense. Namely, $u$ and $\{\Gamma_{s(t)}\}_{t\in [T_0,T_1]}$ satisfy
\begin{alignat}{2}
    & \ddot{u}(t) - \dive\big{(}(\C + \V) Eu(t)\big{)} + \dive\Big{(} \int^t_{T_0} \textnormal{e}^{\tau-t}\,\V Eu(\tau) \, \textnormal{d}\tau \Big{)} = \ell(t) \qquad &&\textnormal{in } \Om\setminus\Gamma_{s(t)},\label{eq:forte1ch2}\\
    &u(t) = u_D(t) \qquad && \text{on $\partial_D\Om$,} \label{eq:forte2ch2} \\
    &\Big{(}(\C + \V) Eu(t)  -  \int^t_{T_0} \textnormal{e}^{\tau-t}\,\V Eu(\tau) \, \textnormal{d}\tau \Big{)}\nu=F(t)\nu\qquad && \text{on $\partial_N\Om$,}\\
    &\Big{(}(\C + \V) Eu(t)  -  \int^t_{T_0} \textnormal{e}^{\tau-t}\,\V Eu(\tau) \, \textnormal{d}\tau \Big{)}^{\!\pm}\!\!\!\nu=F(t)^\pm\nu\qquad && \text{on $\Gamma_{s(t)}$,}\label{eq:bcNCrackch2}\\
    &u(T_0)=u^0\quad\textnormal{and}\quad\dot u(T_0)=u^1\quad && {}\label{eq:forte5ch2} 
    \end{alignat}
    for every $t\in[T_0,T_1]$, where $\ell(t):=f(t)-\dive F(t)$, $\nu$ is the unit normal, and the symbol $\pm$ in \eqref{eq:bcNCrackch2} denotes suitable limits on each side of $\Gamma_{s(t)}$.
\end{remark}

Existence of the solution for the viscoelastic problem \eqref{eq:weak_generalizzata1-t0t1}-\eqref{eq:weak_generalizzata3-t0t1} is given by \cite{Sapio} for $\Omega \subset \R^d$ with $d\geq 1$ and under more general assumptions on the regularity of the cracks. Uniqueness and continuous dependence on the data are proved in \cite{Cianci-Dalmaso} under the assumption that the constant $\mu$, which controls the speed of the crack tip in Definition \ref{def:lunghezze}, satisfies
\begin{equation}\label{eq:vecchio-controllo}
    0< \mu < \mu_0,
\end{equation}
where the constant $\mu_0$ is not explicitly defined in terms of the data of the problem. Using the fact that $d=2$ in our work, we will prove that uniqueness and continuous dependence can be obtained under the explicit assumption
\begin{equation}\label{eq:controllo-velocità}
    0 < \mu < \sqrt{\lambda}/2,
\end{equation}
where $\lambda$ are the constants that appears in Defintion \ref{def:spazio-tensori} respectively.

In order to prove this results, we have to define an auxiliary problem, which can be interpreted as the elastodynamics problem with elasticity tensor replaced by $\A$.

\begin{definition}[Solution for elastodynamics with cracks]\label{def:eq-onder-t0t1}
 Let $\gamma\in \mathcal{G}_{r,L}$, $0\leq T_0 < T_1 \leq T$, $s \in \mathcal{S}^{piec}_{\mu,M}(T_0,T_1)$, with $s(T_1)\leq b_\gamma$, and assume \eqref{eq:C-V}, \eqref{eq:forzante-f-t0t1}-\eqref{eq:condizione-nulla-dentro}.  Let
$ u^{0} \in V_{s(T_0)}$, such that $u^{0}-u_D(T_0)\in V^D_{s(T_0)}$ and let $u^{1}\in H$. We say that $v$ is a weak solution of the problem of elastodynamics on the cracked domains $\Omega\setminus\Gamma_{s(t)}$, $t\in[T_0,T_1]$, with initial conditions $u^{0}$ and $u^{1}$, if
\begin{align}
    &\label{eq:onde1-t0t1} \,\,\,v\in\mathcal{V}(T_0,T_1) \quad \text{and} \quad v-u_D\in \mathcal{V}^D(T_0,T_1),\\
    &-\int^{T_1}_{T_0}(\dot v(t),\dot \varphi(t))\,\dt+ \int^{T_1}_{T_0} (\A Ev(t),E\varphi(t)) \,\dt = \int^{T_1}_{T_0} (f(t),\varphi(t))\,\dt{} \vspace{1cm}\nonumber \\
    &\label{eq:onde2-t0t1} + \int^{T_1}_{T_0}  (F(t),E\varphi(t))  \,\dt{}  \quad\text{for all }  \varphi\in \mathcal{V}^D(T_0,T_1)  \text{ with } \varphi(T_0)=\varphi(T_1)=0,\\
    & \label{eq:onde3-t0t1}\,\,\,v(T_0)=u^{0} \quad \text{in }  H \quad  \text{and} \quad \dot v(T_1)=u^{1} \quad \text{in } (V^D_{s(T_0)})^*,
\end{align}
\end{definition}

\begin{remark}
 In the case of smooth functions problem \eqref{eq:onde1-t0t1}-\eqref{eq:onde3-t0t1} is satisfied in a stronger sense. Namely, $v$ and $\{\Gamma_{s(t)}\}_{t\in [T_0,T_1]}$ satisfy
\begin{alignat}{2}
    & \ddot{v}(t) - \dive\big{(}\A Ev(t)\big{)} = \ell(t) \qquad &&\textnormal{in } \Om\setminus\Gamma_{s(t)},\label{eq:forte1ch2onde}\\
    &v(t) = u_D(t) \qquad && \text{on $\partial_D\Om$,} \label{eq:forte2ch2onde} \\
    &(\A Ev(t)  )\nu=F(t)\nu\qquad && \text{on $\partial_N\Om$,}\\
    &(\A Ev(t)  )^\pm\nu=F(t)^\pm\nu\qquad && \text{on $\Gamma_{s(t)}$,}\label{eq:bcNCrackch2onde}\\
    &v(T_0)=u^0\quad\textnormal{and}\quad\dot v(T_0)=u^1\quad && {}\label{eq:forte5ch2onde} 
    \end{alignat}
    for every $t\in[T_0,T_1]$, where $\ell(t):=f(t)-\dive F(t)$, $\nu$ is the unit normal, and the symbol $\pm$ in \eqref{eq:bcNCrackch2onde} denotes suitable limits on each side of $\Gamma_{s(t)}$.
\end{remark}

Existence and uniqueness for the system of elastodynamics with cracks \eqref{eq:onde1-t0t1}-\eqref{eq:onde3-t0t1} under the assumption \eqref{eq:controllo-velocità} is given by \cite{Dalmaso-Larsen-Toader}, where the authors consider a slight different formulation of the problem which is stronger in time. The proof, which is based on a localization argument, works also for the formulation given in Definition \ref{def:eq-onder-t0t1}. Then we can state the following result.

\begin{theorem}\label{thm:ex-uniq-onde-t0t1}
Let $\gamma\in \mathcal{G}_{r,L}$, $0\leq T_0 < T_1 \leq T$, $s \in \mathcal{S}^{piec}_{\mu,M}(T_0,T_1)$, with $s(T_1)\leq b_\gamma$, and assume \eqref{eq:C-V}, \eqref{eq:forzante-f-t0t1}-\eqref{eq:condizione-nulla-dentro} and \eqref{eq:controllo-velocità}. Let
$ u^{0} \in V_{s(T_0)}$, such that $u^{0}-u_D(T_0)\in V^D_{s(T_0)}$ and let $u^{1}\in H$. Then there exists a unique solution $v$ of problem \eqref{eq:onde1-t0t1}-\eqref{eq:onde3-t0t1}. Moreover $ v \in \mathcal{V}^\infty(T_0,T_1)$, $v\in C^0_w([T_0,T_1];V)$, and $\dot v\in C^0_w([T_0,T_1];H)$.
\end{theorem}

With the following result we obtain a better regularity with respect to time.

\begin{proposition}\label{rem:equivalenza}
Under the same assumption of Theorem \ref{thm:ex-uniq-onde-t0t1}, let $v$ be the unique solution of problem \eqref{eq:onde1-t0t1}-\eqref{eq:onde3-t0t1}. Then $v\in C^0([T_0,T_1], V) \cap C^1([T_0,T_1], H)$.
\end{proposition}

\begin{proof}
In the case $F=0$, a solution for the elastodynamics with cracks in the sense of \cite{Dalmaso-Larsen-Toader} is also a solution in the sense of Definition \ref{def:eq-onder-t0t1}. By uniqueness, the two solutions coincide. In particular, we get that, if $F=0$, the solution is in $C^0([T_0,T_1], V) \cap C^1([T_0,T_1], H)$.

If the forcing term $F$ is not zero, we can use same approximation argument used in \cite[Lemma 5.7]{Cianci-Dalmaso}. Then for every $\varepsilon >0$ there exists $F_\varepsilon \in H^1(0,T, \tilde{H})$ such that $F_\varepsilon(t)\in C^\infty_c(\Omega\setminus \Gamma;\,\mathbb{R}^{d\times d}_{sym})$ for every $t\in [0,T]$ and
\begin{equation}\label{eq:F_eps_to_F-t0t1}
\|F_\varepsilon - F\|_{L^\infty(0,T;\tilde{H})} + \|\dot F_\varepsilon - \dot F\|_{L^2(0,T;\tilde{H})} < \varepsilon.
\end{equation}
We define $v_\varepsilon$ as the solution of the elastodynamic problem in Definition \ref{def:eq-onder-t0t1} with $F$ replaced by $F_\varepsilon$. Since $F_\varepsilon$ is regular in space we have that
\begin{equation}
    (F_\varepsilon(t),E\psi)=-(\dive F_\varepsilon(t),\psi)
\end{equation}
for all $t\in [0,T]$ and for all $\psi\in V$. It follows that $v_\varepsilon$ is a solution in the sense of Definition \ref{def:eq-onder-t0t1} with $f$ and $F$ respectively replaced by $f-\dive F_\varepsilon$ and $0$. By the results of \cite{Dalmaso-Larsen-Toader} we have that $v_\varepsilon\in C^0([T_0,T_1], V) \cap C^1([T_0,T_1], H)$. Using the continuous dependence on the forcing terms given by \cite[Proposition 4.5]{Cianci-Dalmaso} and \eqref{eq:F_eps_to_F-t0t1}, we obtain that
\begin{equation*}
    \sup_{t \in [0,T]}\|v_\varepsilon(t)- v(t)\|_V + \sup_{t \in [0,T]}\| \dot{v}_\varepsilon(t) - \dot{v}(t) \| \to 0 \quad \textnormal{as } \varepsilon \to 0.
\end{equation*}
In particular, we get that $v\in C^0([T_0,T_1], V) \cap C^1([T_0,T_1], H)$.
\end{proof}

We now fix the notation that will be useful in order to give the main results concerning continuous dependence on the data.\vspace{0,3 cm}

Let $0\leq T_0 < T_1 \leq T$, let $\gamma_k\in \mathcal{G}_{r,L}$ be a sequence of cracks paths, and let $s_k \in \mathcal{S}^{piec}_{\mu,M}(T_0,T_1)$, with $s_k(T_1)\leq b_{\gamma_k}$, be a sequence of crack lengths, we define $V^{\gamma_k}$, $\|\cdot\|_{V^{\gamma_k}}$, $V^{\gamma_k}_{s_k(t)}$, $V^{{\gamma_k},D}_{s_k(t)}$, $\mathcal{V}_{\gamma_k,s_k}(T_0,T_1)$, $\|\cdot\|_{\mathcal{V}_{\gamma_k.s_k}}$, $\mathcal{V}^{D}_{\gamma_k,s_k}(T_0,T_1)$ as in \eqref{eq:VeH-t0t1}-\eqref{eq:VcorsivoD-t0t1} with $\Gamma$ and $\Gamma_{s(t)}$ replaced by $\Gamma^{\gamma_k}:=\gamma_k([a_0,b_{\gamma_k}])$ and $\Gamma^{\gamma_k}_{s_k(t)}:=\gamma_k([a_0,s_k(t)]).$

Let $u^0_k\in V^{\gamma_k}_{s_k(T_0)}$, with $u^{0}_k-u_D(T_0)\in V^{\gamma_k,D}_{s_k(T_0)}$, $u^1_k,\in H$,
\begin{equation}\label{eq:fk-Fk}
f_k\in L^2(0,T; H) \quad\textnormal{and}\quad F_k \in H^1(0,T; \tilde{H}).
\end{equation}

We define $u_k$ as the weak solution of $k$-th viscoelastic problem on the cracked domains $\Omega\setminus\Gamma^{\gamma_k}_{s_k(t)}$, $t\in[T_0,T_1]$, that is
\begin{align}
    &\label{eq:visco1-n-t0t1} \,\,\,u_k\in\mathcal\mathcal{V}_{\gamma_k,s_k}(T_0,T_1) \quad \text{and} \quad u_k-u_D\in \mathcal{V}^D_{\gamma_k,s_k}(T_0,T_1),\\
    &-\int^{T_1}_{T_0}(\dot u_k(t),\dot \varphi(t))\,\dt+ \int^{T_1}_{T_0} (\A Eu_k(t),E\varphi(t)) \,\dt \nonumber\\
    &- \int^{T_1}_{T_0}\int^t_{T_0}  \textnormal{e}^{\tau-t} (\V Eu_k(\tau), E\varphi(t))\,\dtau\dt = \int^{T_1}_{T_0} (f_k(t),\varphi(t))\,\dt{}\nonumber\\
    &+ \int^{T_1}_{T_0}  (F_k(t),E\varphi(t))  \,\dt{}\label{eq:visco2-n-t0t1}  \quad\text{for all }  \varphi\in \mathcal{V}^D_{\gamma_k,s_k}(T_0,T_1)  \text{ with } \varphi(T_0)=\varphi(T_1)=0,\\
    & \label{eq:visco3-n-t0t1}\,\,\,u_k(T_0)=u^0_k \quad \text{in }  H \quad  \text{and} \quad \dot u_k(T_1)=u^1_k \quad \text{in } (V^{\gamma_k,D}_{s_k(T_0)})^*.
\end{align}

Moreover, we define $v_k$ as the weak solution of $k$-th problem of elastodynamics on the cracked domains $\Omega\setminus\Gamma^{\gamma_k}_{s_k(t)}$, $t\in[T_0,T_1]$, that is
\begin{align}
    &\label{eq:onde1-n-t0t1} \,\,\,v_k\in\mathcal{V}_{\gamma_k,s_k}(T_0,T_1) \quad \text{and} \quad v_k-u_D\in \mathcal{V}^D_{\gamma_k,s_k}(T_0,T_1),\\
    &-\int^{T_1}_{T_0}(\dot v_k(t),\dot \varphi(t))\,\dt+ \int^{T_1}_{T_0} (\A Ev_k(t),E\varphi(t)) \,\dt = \int^{T_1}_{T_0} (f_k(t),\varphi(t))\,\dt{} \vspace{1cm}\nonumber \\
    &\label{eq:onde2-n-t0t1} + \int^{T_1}_{T_0}  (F_k(t),E\varphi(t))  \,\dt{}  \quad\text{for all }  \varphi\in \mathcal{V}_{\gamma_k,s_k}(T_0,T_1)  \text{ with } \varphi(T_0)=\varphi(T_1)=0,\\
    & \label{eq:onde3-n-t0t1}\,\,\,v_k(T_0)=u^0_k \quad \text{in }  H \quad  \text{and} \quad \dot v_k(T_1)=u^1_k \quad \text{in } (V^{\gamma_k,D}_{s_k(T_0)})^*.
\end{align}

We now state the result concernig continuous dependence on the data for the problem of elastodynamics. It will be used to prove the same result for the viscoelastic problem.

\begin{theorem}\label{thm:dip-cont-fract+c.i.-t0t1} Let $\gamma\in \mathcal{G}_{r,L}$, $0\leq T_0 < T_1 \leq T$, $s \in \mathcal{S}^{piec}_{\mu,M}(T_0,T_1)$, with $s(T_1)\leq b_\gamma$, and assume \eqref{eq:C-V}, \eqref{eq:forzante-f-t0t1}-\eqref{eq:condizione-nulla-dentro} and \eqref{eq:controllo-velocità}. Let $ u^{0} \in V^\gamma_{s(T_0)}$, with $u^{0}-u_D(T_0)\in V^{\gamma,D}_{s(T_0)}$ and let $u^{1}\in H$.
Let $\gamma_k\in \mathcal{G}_{r,L}$, let $s_k \in \mathcal{S}^{piec}_{\mu,M}(T_0,T_1)$, with $s_k(T_1)\leq b_{\gamma_k}$. Let $u^0_k\in V^{\gamma_k}_{s_k(T_0)}$, with $u^{0}_k-u_D(T_0)\in V^{\gamma_k,D}_{s_k(T_0)}$, $u^1_k,\in H$, and assume \eqref{eq:fk-Fk}. Let $v$ be the weak solution of problem \eqref{eq:onde1-t0t1}-\eqref{eq:onde3-t0t1} on the cracked domains $\Omega\setminus\Gamma^{\gamma}_{s(t)}$, $t\in[T_0,T_1]$. Let $v_k$ the weak solution problem \eqref{eq:onde1-n-t0t1}-\eqref{eq:onde3-n-t0t1} on the cracked domains $\Omega\setminus\Gamma^{\gamma_k}_{s_k(t)}$, $t\in[T_0,T_1]$. Assume that
\begin{equation}\label{eq:conf-forz-t0t1}
    \|f_k - f\|_{\Ls(0,T;H)}\to 0,
\qquad
    \|F_k - F\|_{ H^1(0,T;\tilde{H})}\to 0,
\end{equation}
\begin{equation}
    s_k \to s \quad \textnormal{uniformly},
\qquad
    \gamma_k \to \gamma \quad \textnormal{uniformly},
\end{equation}
\begin{equation}
    u^0_k \to u^0 \quad \textnormal{in } H,
\qquad
    Du^0_k \to Du^0 \quad \textnormal{in } \underline{H},
\qquad
    u^1_k \to u^1 \quad \textnormal{in } H.
\end{equation}
Then
\begin{equation}
    v_k(t) \to v(t) \quad \textnormal{in } {H},
\end{equation}
\begin{equation}
     Dv_k(t) \to Dv(t) \quad \textnormal{in } \underline{H},
\end{equation}
\begin{equation}
    \dot v_k(t) \to \dot v(t) \quad \textnormal{in } {H},
\end{equation}
for every $t\in [T_0,T_1]$.
\end{theorem}

\begin{proof}
In the case $f_k=f$, $F_k=F=0$ for any $k\in \mathbb{N}$, it is a consequence of \cite[Theorem 3.5]{Dalmaso-Larsen-Toader}. In the general case, the result follows from the same approximation argument used in \cite[Lemma 5.7, Proposition 5.9]{Cianci-Dalmaso}.
\end{proof}

Now we are in a position to obtain the same results for the viscoelastic system.

\begin{theorem}\label{thm:ex-uniq-visco-t0t1}
Let $\gamma\in \mathcal{G}_{r,L}$, $0\leq T_0 < T_1 \leq T$, $s \in \mathcal{S}^{piec}_{\mu,M}(T_0,T_1)$, with $s(T_1)\leq b_\gamma$, and assume \eqref{eq:C-V}, \eqref{eq:forzante-f-t0t1}-\eqref{eq:condizione-nulla-dentro} and \eqref{eq:controllo-velocità}. Let
$ u^{0} \in V_{s(T_0)}$, such that $u^{0}-u_D(T_0)\in V^D_{s(T_0)}$ and let $u^{1}\in H$. Then there exists a unique solution $u$ of problem \eqref{eq:weak_generalizzata1-t0t1}-\eqref{eq:weak_generalizzata3-t0t1}. Moreover $ u \in \mathcal{V}^\infty(T_0,T_1)$, $u\in C^0_w([T_0,T_1];V)$, and $\dot u\in C^0_w([T_0,T_1];H)$.
\end{theorem}
\begin{proof}
We can not apply directly \cite[Theorem 4.10]{Cianci-Dalmaso} because in general \eqref{eq:vecchio-controllo} is not satisfied. However, assuming \eqref{eq:controllo-velocità} instead of \eqref{eq:vecchio-controllo} we can repeat all arguments of the proof of that theorem, which is based on existence and uniqueness for elastodynamics with cracks (in our case given by Theorem \ref{thm:ex-uniq-onde-t0t1} and Theorem \ref{thm:dip-cont-fract+c.i.-t0t1}) and on a fixed point argument.
\end{proof}

\begin{proposition}\label{rem:reg-viscoelastic}
Under the same assumptions of Theorem \ref{thm:ex-uniq-visco-t0t1}, let $u$ be the unique solution of problem \eqref{eq:weak_generalizzata1-t0t1}-\eqref{eq:weak_generalizzata3-t0t1}. Then $u\in C^0([T_0,T_1], V) \cap C^1([T_0,T_1], H)$.
\end{proposition}
\begin{proof}
It is enough to apply Proposition \ref{rem:equivalenza} with $F(t)$ replaced by
\begin{equation*}
    F(t)+\int^t_{T_0} \textnormal{e}^{\tau-t}\V Eu(\tau) \dtau,
\end{equation*}
for all $t\in [T_0,T_1]$.
\end{proof}

The following theorem provides the continuous dependence on the data for the solution of the viscoelastic problem.

\begin{theorem}\label{thm:dip-cont-fract+c.i.-VISCO-t0t1}
Let $\gamma\in \mathcal{G}_{r,L}$, $0\leq T_0 < T_1 \leq T$, $s \in \mathcal{S}^{piec}_{\mu,M}(T_0,T_1)$, with $s(T_1)\leq b_\gamma$, and assume \eqref{eq:C-V}, \eqref{eq:forzante-f-t0t1}-\eqref{eq:condizione-nulla-dentro} and \eqref{eq:controllo-velocità}. Let $ u^{0} \in V^\gamma_{s(T_0)}$, such that $u^{0}-u_D(T_0)\in V^{\gamma,D}_{s(T_0)}$ and let $u^{1}\in H$.
Let $\gamma_k\in \mathcal{G}_{r,L}$, let $s_k \in \mathcal{S}^{piec}_{\mu,M}(T_0,T_1)$, with $s_k(T_1)\leq b_{\gamma_k}$. Let $u^0_k\in V^{\gamma_k}_{s_k(T_0)}$, such that $u^{0}_k-u_D(T_0)\in V^{\gamma_k,D}_{s_k(T_0)}$, $u^1_k,\in H$, and assume \eqref{eq:fk-Fk}. Let $u$ be the weak solution of problem \eqref{eq:weak_generalizzata1-t0t1}-\eqref{eq:weak_generalizzata3-t0t1} on the cracked domains $\Omega\setminus\Gamma^\gamma_{s(t)}$, $t\in[T_0,T_1]$. Let $u_k$ the weak solution problem \eqref{eq:visco1-n-t0t1}-\eqref{eq:visco3-n-t0t1} on the cracked domains $\Omega\setminus\Gamma^{\gamma_k}_{s_k(t)}$, $t\in[T_0,T_1]$. Assume that
\begin{equation}\label{eq:conf-forz-visc-t0t1}
    \|f_k - f\|_{\Ls(0,T;H)}\to 0,
\qquad
    \|F_k - F\|_{ H^1(0,T;\tilde{H})}\to 0,
\end{equation}
\begin{equation}
    s_k \to s \quad \textnormal{uniformly},
\qquad
    \gamma_k \to \gamma \quad \textnormal{uniformly},
\end{equation}
\begin{equation}
    u^0_k \to u^0 \quad \textnormal{in } H,
\qquad
    Du^0_k \to Du^0 \quad \textnormal{in } \underline{H},
\qquad
    u^1_k \to u^1 \quad \textnormal{in } H.
\end{equation}
Then
\begin{equation}
    u_k(t) \to u(t) \quad \textnormal{in } {H},
\end{equation}
\begin{equation}
     Du_k(t) \to Du(t) \quad \textnormal{in } \underline{H},
\end{equation}
\begin{equation}
    \dot u_k(t) \to \dot u(t) \quad \textnormal{in } {H},
\end{equation}
for every $t\in [T_0,T_1]$. Moreover there exists a constant $C>0$ such that
\begin{equation*}
    \| u_k(t) \| + \| D u_k(t)\| + \| \dot{ u}_k(t) \| \leq C
\end{equation*}
for every $k\in \mathbb{N}$ and $t \in [T_0,T_1]$.
\end{theorem}
\begin{proof}
As in the proof of Theorem \ref{thm:ex-uniq-visco-t0t1}, we cannot apply directly \cite[Theorem 6.1]{Cianci-Dalmaso}, because in general \eqref{eq:vecchio-controllo} is not satisfied. However, assuming \eqref{eq:controllo-velocità} instead of \eqref{eq:vecchio-controllo} we can repeat all arguments of the proof of that theorem, which is based on the continuous dependence on the data for elastodynamics with cracks (in our case given by Theorem \ref{thm:dip-cont-fract+c.i.-t0t1}) and on a results concerning the convergence of fixed points of a sequence of functions (see \cite[Lemma 4.2]{Cianci-Dalmaso}).
\end{proof}

\section{Energy balance}\label{sect:energy-balance}

In this section we study the problem of the dynamic energy-dissipation balance on a given cracked domain $\Omega\setminus\Gamma^\gamma_{s(t)}$ for a solution of a viscoelastic problem.

Let $\gamma\in \mathcal{G}_{r,L}$, $0\leq T_0 < T_1 \leq T$, $s \in \mathcal{S}^{piec}_{\mu,M}(T_0,T_1)$, with $s(T_1)\leq b_\gamma$. It is convenient to define the operator $\mathcal{L}_{T_0}\colon\mathcal{V}(T_0,T_1)\to H^1(T_0,T_1;H)$ as
\begin{equation}\label{eq:def-L_T-t0t1}
    (\mathcal{L}_{T_0}u)(t):= \int^t_{T_0} \textnormal{e}^{\tau-t}\V Eu(\tau) \dtau,
\end{equation}
for all $u\in \mathcal{V}(T_0,T_1)$, for all $t\in [T_0,T_1]$.
Since
\begin{equation*}
     (\dot{\wideparen{\mathcal{L}_{T_0}u}})(t)=  \V Eu(t) -  \int^t_{T_0}  \textnormal{e}^{\tau-t} \V Eu(\tau) \,\dtau,
\end{equation*}
it is easy to check that $\mathcal{L}_{T_0}$ is bounded. Indeed, using the H\"{o}lder inequality it is possible to prove that
\begin{equation}\label{eq:Lu-Linfty-t0t1}
    \|\mathcal{L}_{T_0}u\|_{L^\infty(T_0,T_1;\tilde{H})} \leq (T_1-T_0)^{1/2} \|\V\|_{\infty} \| u \|_{\mathcal{V}(T_0,T_1)},
\end{equation}
\begin{equation}\label{eq:dotLu-L2-t0t1}
    \|\dot{\wideparen{\mathcal{L}_{T_0}u}}\|_{L^2(T_0,T_1;\tilde{H})} \leq (1+T_1-T_0) \|\V\|_{\infty} \| u \|_{\mathcal{V}(T_0,T_1)}.
\end{equation}

Assume \eqref{eq:C-V}, \eqref{eq:forzante-f-t0t1}-\eqref{eq:condizione-nulla-dentro} and let $v \in C^0([T_0,T_1],V)\cap C^1([T_0,T_1],H)$. For every $t \in [T_0,T_1]$ the sum of kinetic and elastic energy is given by
\begin{equation}\label{eq:EnCinEl}
    \mathcal{E}_v(t)=\frac{1}{2}\|\dot v(t)\|^2+ \frac{1}{2}(\C Ev(t),Ev(t)).
\end{equation}
For an interval $[t_1,t_2]\subset [T_0,T_1]$ the dissipation due to viscosity between time $t_1$ and $t_2$ is given by
\begin{align}\label{eq:ViscoDiss}
    \mathcal{D}_v(t_1,t_1)=&\frac{1}{2}(\V Ev(t_2),Ev(t_2)) - \frac{1}{2}(\V Ev(t_1),Ev(t_1))\nonumber\\ &-((\mathcal{L}_{T_0}v)(t_2),Ev(t_2)) +((\mathcal{L}_{T_0}v)(t_1),Ev(t_1))\nonumber\\
    &+\int^{t_2}_{t_1} (\V Ev(t),Ev(t)) \dt{} - \int^{t_2}_{t_1}((\mathcal{L}_{T_0}v)(t),Ev(t)) \dt{}.
\end{align}
Moreover, we assume that the energy dissipated in the process of crack production on the interval $[t_1,t_2]$ is proportional to  $s(t_2)-s(t_1)$, which represent the length of the crack increment. For simplicity we take the proportionality constant equal to one. 
Finally, the work done between time $t_1$ and $t_2$ by the boundary and volume forces is
\begin{align}\label{eq:WorkEqWeak}
    \!\!\mathcal{W}_v(t_1,t_2)\!=&\! \int^{t_2}_{t_1} \Big( (f(t), \dot v(t) - \dot u_D(t)) + ((\C+\V)Ev(t),E\dot{u}_D(t) ) - ((\mathcal{L}_{T_0}v)(t),E\dot{u}_D(t) )  \Big)\dt{} \nonumber \\
    -&\int^{t_2}_{t_1}(\dot F(t), Ev(t)- Eu_D(t))\dt{} -\int^{t_2}_{t_1} (\dot v(t), \ddot v_D(t)) \dt{} + (\dot v(t_2), \dot u_D(t_2)) \nonumber \\
    -& (\dot v(t_1), \dot u_D(t_1)) \!+\! (F(t_2), Ev(t_2)-Eu_D(t_2))\!-\!(F(t_1), Ev(t_1)-Eu_D(t_1)).
\end{align}

\begin{remark}
 When $F=F_0$ as in \eqref{F_0} and all terms are regular enough, formulas \eqref{eq:ViscoDiss} and \eqref{eq:WorkEqWeak} can be obtained from \eqref{intro-mainprobl-nc-infty} in $(-\infty, T]$, using the explicit expression of the stress tensor \eqref{stress-orig} and integrating by parts. For more details when viscosity is not present see also to \cite[Section 3]{Dalmaso-Larsen-Toader2016} and \cite[Section 4]{Dalmaso-Larsen-Toader}.
\end{remark}

\begin{remark}
 We stress that \eqref{eq:ViscoDiss} and \eqref{eq:WorkEqWeak} make sense for every weak solution of problem \eqref{eq:weak_generalizzata1-t0t1}-\eqref{eq:weak_generalizzata2-t0t1}, thanks to Proposition \ref{rem:reg-viscoelastic}.
\end{remark}

We now define the class of cracks whose solutions of the viscoelastic problem satisfy the dynamic energy-dissipation balance.

\begin{definition}\label{def:Creg-t0t1}
Let $0\leq T_0 < T_1 \leq T$, $s_0 \geq 0$, and $\overline{\gamma}\in \mathcal{G}_{r,L}$, with $b_{\overline{\gamma}}=s_0$, and assume \eqref{eq:C-V}, \eqref{eq:forzante-f-t0t1}-\eqref{eq:condizione-nulla-dentro} and \eqref{eq:controllo-velocità}. Let
$ u^{0} \in V^{\overline{\gamma}}_{s_0}$, such that $u^{0}-u_D(T_0)\in V^{\overline{\gamma},D}_{s_0}$ and let $u^{1}\in H$. The class $$\mathcal{B}^{reg}(T_0,T_1)=\mathcal{B}^{reg}(T_0,T_1,s_0,\overline{\gamma},\C,\V,f,F,u_D,u^0,u^1)$$ is composed of all pairs $(\gamma,s)$, with $\gamma\in \mathcal{G}_{r,L}$, $\gamma|_{[a_0,s_0]}=\overline{\gamma}|_{[a_0,s_0]}$, $s\in \mathcal{S}^{reg}_{\mu,M}([T_0,T_1])$, $s(T_0)=s^0$, and $s(T_1)\leq b_\gamma$, such that the unique weak solution $u$ of the viscoelastic problem \eqref{eq:weak_generalizzata1-t0t1}-\eqref{eq:weak_generalizzata3-t0t1} satisfies the energy-dissipation balance
\begin{align}\label{eq:energy-diss-balance-t0t1}
\mathcal{E}_u(t_2)-\mathcal{E}_u(t_1) + \mathcal{D}_u(t_1,t_2) + s(t_2) - s(t_1) = \mathcal{W}_u(t_1,t_2)
\end{align}
for every interval $[t_1,t_2]\subset [T_0,T_1]$. Similarly, the class $$\mathcal{B}^{piec}(T_0,T_1)= \mathcal{B}^{piec}(T_0,T_1,s_0,\overline{\gamma},\C,\V,f,F,u_D,u^0,u^1)$$
is defined in the same way replacing $s\in \mathcal{S}^{reg}_{\mu,M}([T_0,T_1])$ by $s\in \mathcal{S}^{piec}_{\mu,M}([T_0,T_1])$.

\end{definition}

The class $\mathcal{B}^{reg}(T_0,T_1)$ is nonempty, as clarified by the following result, whose proof follows the lines of \cite[Lemma 1]{Dalmaso-Scala} and \cite[Proposition 2.7]{Dalmaso-Sapio}.

\begin{proposition}\label{prop:frattura-fissa}
Under the assumption of Definition \ref{def:Creg-t0t1}, the pair $(\overline{\gamma}, s)$, with $s(t)=s_0$ for every $t\in [T_0,T_1]$, belongs to $\mathcal{B}^{reg}(T_0,T_1)$.
\end{proposition}
\begin{proof}
We prove the result in the case of homogeneous boundary condition, i.e. $u_D=0$. Indeed, the case of non-homogeneous data can be obtained considering the equation for $u-u_D$. It is convenient to extend our data on $[0,2T]$ by setting $f(t)=0$ and $F(t)=F(T)$ for $t\in (T,2T]$. It is clear that $f\in L^2(0,2T, H)$, $F\in H^1(0,2T,\tilde{H})$, and that, by uniqueness, the solution $u$ of the viscoelastic problem on $[T_0,2T]$ is an extension of the solution on $[T_0,T_1]$.
Since the domain is constant with respect to time we deduce from \eqref{eq:weak_generalizzata1-t0t1}-\eqref{eq:weak_generalizzata2-t0t1} that $u\in H^2([T_0,2T]; (V^D_{s_0})^*)$ and
\begin{equation}\label{eq:wstrongt0t1}
    \langle \ddot{u} (t), \varphi \rangle + ((\C+\V)Eu(t),E\varphi)-(\mathcal{L}_{T_0}u(t), E\varphi)=(f(t),\varphi) + (F(t), E\varphi).
\end{equation}
for all $\varphi \in V^D_{s_0}$ and for a.e. $t\in [T_0,2T]$.

Given a Banach space $X$ and a function $r:[T_0,2T] \to X$, for every $h>0$ we define $\sigma^hr,\delta^hr:[T_0,2T-h]\to X$ by $\sigma^hr(t):=r(t+h)+r(t)$, $\delta^hr(t):=r(t+h)-r(t)$. For a.e. $t\in [T_0,2T-h]$ we have $\sigma^hu(t),\delta^hu(t) \in V^D_{s_0}.$ We consider \eqref{eq:wstrongt0t1} at time $t$ and a time $t+h$, in both cases with $\varphi=\delta^h u(t)$. We sum the two expressions and we integrate on $[t_1,t_2] \subseteq [T_0,T_1]$. We get 
\begin{equation}\label{eq:diff-finite-t0t1}
     \int^{t_2}_{t_1} \big( K_h (t) + E_h(t) + D_h(t) \big) \, \dt = \int^{t_2}_{t_1} L_h(t) \, \dt{},
\end{equation}
where the terms that appear in \eqref{eq:diff-finite-t0t1} are defined as
\begin{align*}
    K_h(t)&:=\langle\sigma^h\ddot{u}(t), \delta^h u(t)\rangle, \\
    E_h(t)&:=((\C+\V)\sigma^hEu(t),\delta^hEu(t)), \\
    D_h(t)&:=-(\sigma^h[\mathcal{L}_{T_0}u(t)], \delta^hEu(t)),\\
    L_h(t)&:=(\sigma^hf(t), \delta^hu(t)) + (\sigma^hF(t), \delta^hEu(t)).
\end{align*}
We have that
\begin{align*}
    \int^{t_2}_{t_1} K_h(t) \, \dt{}&=- \int^{t_2}_{t_1} (\sigma^h\dot{u}(t), \delta^h \dot u(t)) \, \dt + (\sigma^h\dot{u}(t_2), \delta^h  u(t_2)) - (\sigma^h\dot{u}(t_1), \delta^h  u(t_1))\nonumber\\
    &=-\int^{t_2}_{t_1} \big( \| \dot{u}(t+h) \|^2 \dt{}- \| \dot{u}(t) \|^2 \big) \dt{} + (\sigma^h\dot{u}(t_2), \delta^h  u(t_2)) - (\sigma^h\dot{u}(t_1), \delta^h  u(t_1))\nonumber\\
    &=-\int^{t_2+h}_{t_1+h} \| \dot{u}(t) \|^2 \dt{} + \int^{t_2}_{t_1} \| \dot{u}(t) \|^2 \dt{} + (\sigma^h\dot{u}(t_2), \delta^h u(t_2)) - (\sigma^h\dot{u}(t_1), \delta^h  u(t_1))\nonumber\\
    &= - \int^{t_2+h}_{t_2}\!\!\!\! \| \dot{u}(t) \|^2 \dt{} + \int^{t_1+h}_{t_1} \!\!\!\! \| \dot{u}(t) \|^2 \dt{} + (\sigma^h\dot{u}(t_2), \delta^h  u(t_2)) - (\sigma^h\dot{u}(t_1), \delta^h  u(t_1)) \nonumber 
\end{align*}
and dividing by $h$ we get
\begin{align*}
    \int^{t_2}_{t_1}\!\frac{K_h(t)}{h}  \, \dt{}&=- \fint^{t_2+h}_{t_2}\!\!\!\!\!\! \| \dot{u}(t) \|^2 \dt{} + \fint^{t_1+h}_{t_1} \!\!\!\!\!\! \| \dot{u}(t) \|^2 \dt{} + (\sigma^h\dot{u}(t_2), \frac{\delta^h u(t_2)}{h}) - (\sigma^h\dot{u}(t_1), \frac{\delta^h u(t_1)}{h}).
\end{align*}
Then
\begin{equation}\label{eq:K_h}
      \int^{t_2}_{t_1}\frac{K_h(t)}{h} \, \dt{}\to - \|\dot u(t_2)\|^2 + \|\dot u(t_1)\|^2 + 2 \|\dot u(t_2)\|^2 - 2 \|\dot u(t_1)\|^2= \|\dot u(t_2)\|^2 - \|\dot u(t_1)\|^2,
\end{equation}
as $h\to 0^+$, where we have used the fact that $u \in  C^1([T_0,2T], H)$.
Moreover
\begin{align}
    \int^{t_2}_{t_1} E_h(t) \, \dt{}&= \int^{t_2}_{t_1} ((\C+ \V)Eu(t+h), Eu(t+h)) \dt{} - \int^{t_2}_{t_1} ((\C+ \V)Eu(t),Eu(t)) \dt{} \nonumber\\
    &=\int^{t_2+h}_{t_1+h} ((\C+ \V)Eu(t), Eu(t)) \dt{} - \int^{t_2}_{t_1} ((\C+ \V)Eu(t),Eu(t)) \dt{} \nonumber\\
    &= \int^{t_2+h}_{t_2} ((\C+ \V)Eu(t), Eu(t)) \dt{} - \int^{t_1+h}_{t_1} ((\C+ \V)Eu(t),Eu(t)) \dt{} 
\end{align}
which give us
\begin{equation}\label{eq:E_h}
    \int^{t_2}_{t_1}\frac{E_h(t)}{h}  \, \dt{} \to ((\C+ \V)Eu(t_2), Eu(t_2)) - ((\C+ \V)Eu(t_1),Eu(t_1))
\end{equation}
as $h \to 0^+$, where we have used the fact that $u \in C^0([T_0,2T], V)$.
Regarding the term $D_h$ we have
\begin{align}
    -\int^{t_2}_{t_1} D_h(t) \, \dt{}&= \int^{t_2}_{t_1} (\sigma^h[\mathcal{L}_{T_0}u(t)],Eu(t+h)) \dt{} - \int^{t_2}_{t_1} (\sigma^h[\mathcal{L}_{T_0}u(t)],Eu(t)) \dt{} \nonumber\\
    &=\int^{t_2+h}_{t_1+h} (\sigma^{-h}[\mathcal{L}_{T_0}u(t)],Eu(t)) \dt{} - \int^{t_2}_{t_1} (\sigma^h[\mathcal{L}_{T_0}u(t)],Eu(t)) \dt{} \nonumber\\
    &= \int^{t_2+h}_{t_1+h} (\mathcal{L}_{T_0}u(t-h) - \mathcal{L}_{T_0}u(t+h),Eu(t)) \dt{} \nonumber \\ &\,\,\,\,\,\,-\int^{t_1+h}_{t_1} (\sigma^h[\mathcal{L}_{T_0}u(t)],Eu(t)) \dt{} + \int^{t_2+h}_{t_2} (\sigma^h[\mathcal{L}_{T_0}u(t)],Eu(t)) \dt{} \nonumber\\
    &= \int^{t_2}_{t_1} (\mathcal{L}_{T_0}u(t) - \mathcal{L}_{T_0}u(t+2h),Eu(t+h)) \dt{} \nonumber \\ &\,\,\,\,\,\,-\int^{t_1+h}_{t_1} (\sigma^h[\mathcal{L}_{T_0}u(t)],Eu(t)) \dt{} + \int^{t_2+h}_{t_2} (\sigma^h[\mathcal{L}_{T_0}u(t)],Eu(t)) \dt{},
\end{align}
which give us
\begin{align}\label{eq:D_h}
    \int^{t_2}_{t_1} \frac{D_h(t)}{h}  \, \dt{} &= -\int^{t_2}_{t_1} \Big(\frac{\mathcal{L}_{T_0}u(t) - \mathcal{L}_{T_0}u(t+2h)}{h},Eu(t+h)\Big) \dt{} \nonumber \\ &\,\,\,\,\,\,+\fint^{t_1+h}_{t_1} (\sigma^h[\mathcal{L}_{T_0}u(t)],Eu(t)) \dt{} - \fint^{t_2+h}_{t_2} (\sigma^h[\mathcal{L}_{T_0}u(t)],Eu(t)) \dt{}\nonumber\\
    &\to 2\int^{t_2}_{t_1} ( (\dot{\wideparen{\mathcal{L}_{T_0}u}})(t),Eu(t)) \dt{} \nonumber \\ &\,\,\,\,\,\,+2(\mathcal{L}_{T_0}u(t_1),Eu(t_1)) - 2(\mathcal{L}_{T_0}u(t_2),Eu(t_2))\nonumber\\
    &= 2\int^{t_2}_{t_1}  (\V Eu(t)-\mathcal{L}_{T_0}u(t),Eu(t)) \dt{} \nonumber \\
    &\,\,\,\,\,\,+2(\mathcal{L}_{T_0}u(t_1),Eu(t_1)) - 2(\mathcal{L}_{T_0}u(t_2),Eu(t_2)), \quad \textnormal{as } h\to 0^+,
\end{align}
where we have used again that $u\in C^0([T_0,2T],V)$.

With similar arguments, we have that
\begin{align}\label{eq:L_h}
    \int^{t_2}_{t_1} \frac{L_h(t)}{h}  \dt{} & \to 2\int^{t_2}_{t_1} (f(t),\dot u(t)) \dt{} -2 \int^{t_2}_{t_1} (\dot F(t), Eu(t)) \dt{} \nonumber\\
    &\,\,\,\,\,\, + 2(\dot F(t_2), Eu(t_1)) - 2(\dot F(t_1), Eu(t_1)), \quad \textnormal{as } h\to 0^+.
\end{align}

Dividing by $h$ Equation \eqref{eq:diff-finite-t0t1} and using Equations \eqref{eq:K_h}, \eqref{eq:E_h}, \eqref{eq:D_h}, and \eqref{eq:L_h}, we get the following identity
\begin{align}
   & \|\dot u(t_2)\|^2 + ((\C+ \V)Eu(t_2), Eu(t_2)) + 2\int^{t_2}_{t_1}  (\V Eu(t)-\mathcal{L}_{T_0}u(t),Eu(t)) \dt{}\nonumber\\
   &- 2(\mathcal{L}_{T_0}u(t_2),Eu(t_2)) = \|\dot u(t_1)\|^2 + ((\C+ \V)Eu(t_1),Eu(t_1)) -2(\mathcal{L}_{T_0}u(t_1),Eu(t_1))\nonumber\\
   & + 2\int^{t_2}_{t_1} (f(t),\dot u(t)) \dt{} -2 \int^{t_2}_{t_1} (\dot F(t), Eu(t)) \dt{} + 2(\dot F(t_2), Eu(t_1)) - 2(\dot F(t_1), Eu(t_1)),
\end{align}
that is the energy-dissipation balance \eqref{eq:energy-diss-balance-t0t1} when $u_D=0$ and $s(t)=s_0$ for all $t\in [T_0,T_1]$.
\end{proof}

The following remark deals with the concatenation of solutions on adjacent time intervals.

\begin{remark}\label{Remark:concatenation}
 Under the assumption of Definition \ref{def:Creg-t0t1}, let $0 \leq T_0 < T_1 < T_2 \leq T$, $$(\gamma_1,s_1)\in \mathcal{B}^{piec}(T_0,T_1,s_0,\overline{\gamma}, \C, \V, f, F, u_D, u^0, u^1),$$ $$(\gamma_2,s_2)\in\mathcal{B}^{piec}(T_1,T_2,s_1(T_1),\gamma_1, \C, \V, f, F, u_D, u(T_1), \dot u(T_1)).$$
 Let $s \colon [T_0,T_2] \to \mathbb{R}$ be defined as
 \begin{equation}
     s(t):=\begin{cases}
     s_1(t) & \text{if } t\in[T_0,T_1],\\
     s_2(t) & \text{if } t\in[T_1,T_2].
     \end{cases}
 \end{equation}
 Then $(\gamma_2,s)\in\mathcal{B}^{piec}(T_0,T_2,s_0,\overline{\gamma}, \C, \V, f, F, u_D, u^0, u^1).$
 \end{remark}
 
 Using the continuous dependence Theorem \ref{thm:dip-cont-fract+c.i.-VISCO-t0t1} we are in a position to prove a compactness result for $\mathcal{B}^{reg}$, which will be useful for the proof of the main result of the paper (see Theorem \ref{thm:main-dissipat}).
 
 \begin{theorem}\label{thm:compat-Breg}
 Under the assumption of Definition \ref{def:Creg-t0t1}, let $(\gamma_k,\,s_k)\in \mathcal{B}^{reg}(T_0,T_1)$. Then there exists a not relabelled subsequence and there exists $(\gamma,s)\in \mathcal{B}^{reg}(T_0,T_1)$ such that $\gamma_k \to \gamma$ uniformly (in the sense of Definition \ref{def:conv-unif}) and $s_k \to s$ in $C^3([T_0,T_1])$.
 \end{theorem}
 \begin{proof}
 By Lemma \ref{lemma:compattezza-gamma} there exists a subsequence (not relabelled) $\gamma_k$ and $\gamma\in\mathcal{G}_{r,L}$ such that $\gamma_k \to \gamma$ uniformly (in the sense of Definition \ref{def:conv-unif}). By Ascoli-Arzelà Theorem there exists $s\in C^3([T_0,T_1])$ and a further subsequence $s_k$ converging to $s$ in $C^3([T_0,T_1])$. Moreover, if we pass to the limit ad $k\to +\infty$ in the conditions in Definition \ref{def:lunghezze} for $s_k$, we get that $s\in \mathcal{S}^{reg}_{\mu,M}([T_0,T_1])$. We defined $u$ as the solution of the viscoelastic problem \eqref{eq:weak_generalizzata1-t0t1}-\eqref{eq:weak_generalizzata3-t0t1} on the time-dependent cracked domain $t\mapsto \Omega\setminus\Gamma^{\gamma}_{s(t)}$ with $t\in[T_0,T_1]$ and we define $u_k$ as the solution of the viscoelastic problem on the time-dependent cracked domain $t\mapsto \Omega\setminus\Gamma^{\gamma_k}_{s_k(t)}$ with $t\in[T_0,T_1]$. Since $(\gamma_k,s_k)\in\mathcal{B}^{reg}(T_0,T_1)$ we have
 \begin{align}\label{eq:bilancio-k}
    &\frac{1}{2}\|\dot u_k(t_2)\|^2 + \frac{1}{2}((\C+\V)Eu_k(t_2),Eu_k(t_2))-(\mathcal{L}_{T_0}u_k(t_2),Eu_k(t_2)) \nonumber\\
    -&\frac{1}{2}\|\dot u_k(t_1)\|^2 - \frac{1}{2}((\C+\V)Eu_k(t_1),Eu_k(t_1))+(\mathcal{L}_{T_0}u_k(t_1),Eu_k(t_1)) \nonumber\\
    -&\int^{t_2}_{t_1} (\V Eu_k(t),Eu_k(t)) \dt{} - \int^{t_2}_{t_1}(\mathcal{L}_{T_0}u_k(t),Eu_k(t)) \dt{} + s_k(t_2) - s_k(t_1) \nonumber\\
    =& \int^{t_2}_{t_1} \Big( (f(t), \dot u_k(t) - \dot u_D(t)) + ((\C+\V)Eu_k(t),E\dot{u}_D(t) ) - (\mathcal{L}_{T_0}u_k(t),E\dot{u}_D(t) )  \Big)\dt{} \nonumber \\
    -&\int^{t_2}_{t_1}(\dot F(t), Eu_k(t)- Eu_D(t))\dt{} + (F(t_2), Eu_k(t_2)-Eu_D(t_2))-(F(t_1), Eu_k(t_1)-Eu_D(t_1)) \nonumber \\
    -&\int^{t_2}_{t_1} (\dot u_k(t), \ddot u_D(t)) \dt{} + (\dot u_k(t_2), \dot u_D(t_2)) - (\dot u_k(t_1), \dot u_D(t_1)),
\end{align}
for every interval $[t_1,t_2]\subset [T_0,T_1]$. Using Theorem \ref{thm:dip-cont-fract+c.i.-VISCO-t0t1} and the bounds \eqref{eq:Lu-Linfty-t0t1}-\eqref{eq:dotLu-L2-t0t1}, we can pass to the limit as $k \to +\infty$ in \eqref{eq:bilancio-k} and we get the energy-dissipation balance \eqref{eq:energy-diss-balance-t0t1} for $u$. This proves that $(\gamma,s)\in\mathcal{B}^{reg}(T_0,T_1)$ and concludes the proof.
\end{proof}

\section{Existence for the coupled problem}\label{sect:existence-coupled}

In this section we prove an existence result for the crack evolution (described by the functions $\gamma$ and $s$). In order to do this we define a maximal dissipation condition (see also \cite{Dalmaso-Larsen-Toader2016} and \cite{Dalmaso-Larsen-Toader}), which forces the crack tip to choose a path which allows for a maximal speed.

\begin{definition}\label{def:maximal-dissip}
Assume \eqref{eq:C-V}, \eqref{eq:forzante-f-t0t1}-\eqref{eq:condizione-nulla-dentro} and \eqref{eq:controllo-velocità}. Let $ u^{0} \in V_{0}$, such that $u^{0}-u_D(0)\in V^D_{0}$, and let $u^{1}\in H$. Given $\eta>0$ we say that $(\gamma,s)\in \mathcal{B}^{piec}(0,T)$ satisfies the $\eta-$maximal dissipation condition on $[0,T]$ if there exists no $(\hat \gamma, \hat s)\in \mathcal{B}^{piec}(0,\tau_1)$, for some $\tau_1\in (0,T]$, such that
\begin{itemize}
\item[(M1)] $sing(\hat{s}) \subset sing({s})$,
\item[(M2)] $\hat{s}(t)=s(t)$ and $\hat{\gamma}(\hat{s}(t))={\gamma}({s}(t))$ for every $t\in [0,\tau_0]$, for some $\tau\in [0,\tau_1)$,
\item[(M3)] $\hat{s}(t)>s(t)$ for every $t\in (\tau_0,\tau_1]$ and $\hat{s}(\tau_1)>s(\tau_1) + \eta$.
\end{itemize}
\end{definition}

\begin{remark}
 We refer to the discussion in \cite[Section 1]{Dalmaso-Larsen-Toader} for some comments on the presence of the parameter $\eta >0$.
\end{remark}

We are now in position to prove the main result of the paper. The proof follows the lines of \cite{Dalmaso-Larsen-Toader2016} and \cite{Dalmaso-Larsen-Toader}, devoted to the case of elastodynamics without viscosity terms.

\begin{theorem}\label{thm:main-dissipat}
Under the assumption of Definition \ref{def:maximal-dissip}, for every $\eta >0$ there exists a pair $(\gamma, s) \in \mathcal{B}^{piec}(0,T)$ satisfying the $\eta$-maximal dissipation condition on $[0,T]$.
\end{theorem}

\begin{proof}
 Let us fix $\eta>0$ and a finite subdivision $0=T_0<T_1<...<T_k=T$ of the time interval $[0,T]$ such that $T_i - T_{i-1} < \frac{\eta}{\mu}$ for every $i\in\{0,1,2,...,k\}$. We will define the solution usong a recursive procedure on each subinterval $[T_{i-1}, T_i]$, for every $i\in\{0,1,2,...,k\}$. In order to define this procedure, we set
 \begin{equation}\label{eq:A1}
     \mathcal{X}_1:=\Big\{ (\gamma,s) \in \mathcal{B}^{piec}(0,T_1,0,\gamma_0,\C,\V,f,F,u_D,u^0,u^1)\,|\, s \in \mathcal{S}^{reg}_{\mu,M}(0,T_1),\, s(0)=0 \Big\},
 \end{equation}
 where $\gamma_0$ is the function that appears in Definition \ref{def:spazioGrL}. By Proposition \ref{prop:frattura-fissa} we have that $(\gamma_0,0)\in \mathcal{X}_1$ and in particular we have $\mathcal{X}_1\neq \emptyset$.  Moreover, we choose $(\gamma_1,s_1)\in \mathcal{X}_1$ such that
 \begin{equation*}
     \int^{T_1}_{T_0} s_1(t)\, \dt{}= \max_{(\gamma,s)\in \mathcal{X}_1}\int^{T_1}_{T_0} s(t)\, \dt{},
\end{equation*}
where the existence of $(\gamma_1, s_1)$ is guaranteed by Lemma \ref{lemma:esistenza-max} below. If $k=1$, we define $(\gamma,s):=(\gamma_1,s_1)$ and we have to prove that this couple satisfies the $\eta$-maximal dissipation condition. Otherwise, we fix $i\in \{2,...,k\}$ and we set
 \begin{align}\label{eq:Ai}
     \mathcal{X}_i:=\Big\{ (\gamma,s) \in \mathcal{B}^{piec}(0,T_i,0,\gamma_0,\C,\V,f,F,u_D,u^0,u^1)\,|\, s|_{[T_{i-1},T_{i}]} \in \mathcal{S}^{reg}_{\mu,M}(T_{i-1},T_{i}),&\nonumber\\ s(t)=s_{i-1}(t), \gamma(s(t))=\gamma_{i-1}(s_{i-1}(t)) \,\forall
     \,t\in [0,T_{i-1}]\Big\}.&
 \end{align}
 We note that $\mathcal{X}_i\neq \emptyset$. Indeed, if we define $\tilde{s}_{i-1}$ as
 \begin{equation*}
     \tilde{s}_{i-1}(t):=\begin{cases}
      s_{i-1}(t) & \textnormal{for } t\in [0,T_{i-1}],\\
      s_{i-1}(T_{i-1}) & \textnormal{for } t\in [T_{i-1}, T_i],
     \end{cases}
 \end{equation*}
 we can apply Proposition \ref{prop:frattura-fissa} and Remark \ref{Remark:concatenation} to obtain $(\gamma_{i-1},\tilde{s}_{i-1})\in \mathcal{X}_i$.
 Assume that the pair $(\gamma_{i-1},s_{i-1})\in \mathcal{X}_{i-1}$ has already been defined, then we choose $(\gamma_i,s_i)\in \mathcal{X}_i$ such that
 \begin{equation}\label{eq:maximal-cond}
     \int^{T_i}_{T_{i-1}} s_i(t)\, \dt{}= \max_{(\gamma,s)\in \mathcal{X}_i}\int^{T_i}_{T_{i-1}} s(t)\, \dt{},
\end{equation}
where the existence of $(\gamma_i, s_i)$ is guaranteed by Lemma \ref{lemma:esistenza-max} below.

We now define $(\gamma,s):=(\gamma_k,s_k)$, where $(\gamma_k,s_k)$ is the the pair defined in the final step of the procedure defined above. It remains to prove that $(\gamma,s)$ satisfies the $\eta$-maximal dissipation condition on the interval $[0,T]$. Assume, by contradiction that there exist $0 \leq \tau_0 <  \tau_1 \leq T$ and $(\hat\gamma,\hat s) \in \mathcal{B}^{piec}(0,\tau_1)$ such that:
\begin{itemize}
    \item[(i)] $sing(\hat s) \subset sing(s) \subset \{ T_1,..., T_{k-1} \}$
    \item[(ii)] $s(t)=\hat s(t)$ and $\gamma(s(t))=\hat\gamma(\hat s(t))$ for every $t \in [0,\tau_0]$,
    \item[(iii)] $s(t) < \hat s(t)$ for every $t\in (\tau_0,\tau_1]$ and $\hat s (\tau_1) > s(\tau_1) + \eta$.
\end{itemize}
Since $\tau_0 < T$, there exists an index $j\in \{1,...,k\}$ such that $\tau_0\in [T_{j-1},T_j)$. We claim that $\tau_1 > T_j$. Indeed, the using the monotonicity of $s$ and the points (ii) and (iii), we have that $\hat s (\tau_1) > s(\tau_1) + \eta \geq s(\tau_0) + \eta = \hat s(\tau_0) + \eta $ and in particular $ \hat s(\tau_1) - \hat s(\tau_0) > \eta $. On the other hand, since $\hat s \in \mathcal{S}^{piec}_{\mu,M}(0,\tau_1)$ we have $\hat s(\tau_1)- \hat s(\tau_0) \leq \mu (\tau_1 - \tau_0)$, which together with the previous inequality give us $\tau_1 - \tau_0 > \eta/\mu$. Since the subdivision of the interval was choosen such that $T_{i-1}- T_i < \eta/\mu$ for every $i\in \{1,...,k\}$, we get that $\tau_1>T_j$.

Using (i) we have that $\hat s|_{[T_{j-1},T_j]} \in \mathcal{S}^{reg}_{\mu,M}(T_{j-1},T_j)$ and taking (ii) into account we get that $(\hat \gamma, \hat s) \in \mathcal{X}_j$. By construction $s=s_j$ on $[T_{j-1},T_j]$, where $s_j$ is the function defined in \eqref{eq:maximal-cond} for $i=j$. As a consequence of (iii) we get $\hat s (t) > s(t) = s_j (t)$ for every $t\in (\tau_0, T_j]$, which contradicts \eqref{eq:maximal-cond}.

\end{proof}

We close this section with the following Lemma used to prove Theorem \ref{thm:main-dissipat}. The proof can be found in \cite[Lemma 5.3]{Dalmaso-Larsen-Toader} with obvious modifications.

\begin{lemma}\label{lemma:esistenza-max}
 For every $i=1,..., k$ there exists $(\gamma_i,s_i) \in \mathcal{X}_i$ such that
 \begin{equation}\label{eq:SupS-tesi}
     \int^{T_i}_{T_{i-1}} s_i(t)\, \dt{}= \max_{(\gamma,s)\in \mathcal{X}_i}\int^{T_i}_{T_{i-1}} s(t)\, \dt{},
\end{equation}
where $\mathcal{X}_i$ is the space defined in \eqref{eq:A1} and \eqref{eq:Ai}.
\end{lemma}

\vspace{1 cm}

\noindent \textsc{Acknowledgements.}
The author wishes to thank Professor Gianni Dal Maso for
having proposed the problem and for many helpful discussions on the topic. This paper is based on work supported by the National Research Project (PRIN  2017) 
``Variational Methods for Stationary and Evolution Problems with Singularities and 
 Interfaces", funded by the Italian Ministry of University and Research. 
The author is member of the {\em Gruppo Nazionale per l'Analisi Ma\-te\-ma\-ti\-ca, la Probabilit\`a e le loro Applicazioni} (GNAMPA) of the {\em Istituto Nazionale di Alta Matematica} (INdAM).

\vspace{1cm}

{
}


\frenchspacing
\begin{thebibliography}{99}

\bibitem{Boltz_1}  L. Boltzmann: {\it Zur Theorie der elastischen Nachwirkung}, Sitzber. Kaiserl.
Akad. Wiss. Wien, Math.-Naturw. Kl. {\bf 70}, Sect. II (1874), 275-300.

\bibitem{Boltz_2} L. Boltzmann, {\it Zur Theorie der elastischen Nachwirkung}, Ann. Phys. u. Chem., {\bf 5} (1878), 430-432.

\bibitem{Caponi} M. Caponi: {\it Linear Hyperbolic Systems in Domains with Growing Cracks}, Milan J. Math. {\bf 85} (2017), 149-185. 

\bibitem{Caponi-tesi} M. Caponi: {\it On some mathematical problems in fracture dynamics}, Ph.D. Thesis SISSA, Trieste, 2019.

\bibitem{Cianci-Dalmaso} F. Cianci, G. Dal Maso: {\it Uniqueness and continuous dependence for a viscoelastic problem with memory in domains with time dependent cracks}, accepted paper at Differential Integral Equations, 2021.

\bibitem{Dafermos} C. Dafermos: {\it Asymptotic stability in viscoelasticity}, Arch. Rational Mech. Anal. {\bf 37} (1970), 297–308.

\bibitem{Dalmaso-Larsen} G. Dal Maso, C.J. Larsen: { \it Existence for wave equations on domains with arbitrary growing cracks}. Atti Accad. Naz. Lincei Rend. Lincei Mat. Appl. {\bf 22} (2011), no. 3, 387–408.

\bibitem{Dalmaso-Larsen-Toader2016} G. Dal Maso, C.J. Larsen, R. Toader: {\it Existence for constrained dynamic Griffith fracture with a weak maximal dissipation condition}, J. Mech. Phys. Solids {\bf 95} (2016), 697–707.

\bibitem{Dalmaso-Larsen-Toader} G. Dal Maso, C.J. Larsen, R. Toader: {\it Existence for elastodynamic Griffith fracture with a weak maximal dissipation condition}, J. Math. Pures Appl. (9) {\bf 127} (2019), 160–191.

\bibitem{DalMaso-Luc}  G. Dal Maso, I. Lucardesi: {\it The wave equation on domains with cracks growing
on a prescribed path: existence, uniqueness, and continuous dependence on the data}, Appl. Math. Res. Express 2017 (2017), 184–241.

\bibitem{Dalmaso-Sapio} G. Dal Maso, F. Sapio: {\it Quasistatic limit of a dynamic viscoelastic model with memory}, Milan J. Math. {\bf 89} (2021), no. 2, 485–522.

\bibitem{Dalmaso-Scala} G. Dal Maso, R. Scala: {\it Quasistatic evolution in perfect plasticity as limit of dynamic processes}, J. Dynam. Differential Equations 26 (2014), no. 4, 915–954.

\bibitem{DalMaso-Toader} G. Dal Maso, R. Toader: {\it On the Cauchy problem for the wave equation on time-dependent domains}, J. Differential Equations {\bf 266} (2019), 3209-3246.

\bibitem{DL_V1} R. Dautray, J.-L. Lions: {\it Mathematical analysis and numerical methods for science and technology. Vol. 1. Physical origins and classical methods}. With the collaboration of Philippe Bénilan, Michel Cessenat, André Gervat, Alain Kavenoky and Hélène Lanchon. Translated from the French by Ian N. Sneddon. With a preface by Jean Teillac. Springer-Verlag, Berlin, 1990.

\bibitem{Dautray-Lions_V5} R. Dautray, J.-L. Lions: { \it Mathematical analysis and numerical methods for science and technology. Vol. 5. Evolution problems I}, With the collaboration of Michel Artola, Michel Cessenat and Hélène Lanchon. Translated from the French by Alan Craig. Springer-Verlag, Berlin, 1992.

\bibitem{D-L_V8} R. Dautray, J.-L. Lions: { \it Jacques-Louis Analyse mathématique et calcul numérique pour les sciences et les techniques. Vol. 8.} (French) [Mathematical analysis and computing for science and technology. Vol. 8] Évolution: semi-groupe, variationnel. [Evolution: semigroups, variational methods] Reprint of the 1985 edition. INSTN: Collection Enseignement. [INSTN: Teaching Collection] Masson, Paris, 1988.

\bibitem{Fab-Gi-Pata} M. Fabrizio, C. Giorgi, V. Pata: {\it A New Approach to Equations with Memory}, Arch. Rational Mech. Anal. 198 (2010), 189-232.

\bibitem{Griffith} A. Griffith: {\it The phenomena of rupture and flow in solids}, Philos. Trans. Roy. Soc.
London Ser. A {\bf 221} (1920) 163-198.

\bibitem{Fab-Morro} M. Fabrizio, A. Morro: {\it Mathematical problems in linear viscoelasticity}. SIAM Studies in Applied Mathematics, 12. Society for Industrial and Applied Mathematics (SIAM), Philadelphia, PA, 1992.

\bibitem{Larsen} C. Larsen: {\it Models for dynamic fracture based on Griffiths criterion}, K. Hackl (Ed.), IUTAM Symp. Variational Concepts with Applications to the Mechanics of Materials, Springer, 2010, pp. 131-140.

\bibitem{Mott} N.F. Mott: {\it Brittle fracture in mild steel plates}, Engineering {\bf 165}, 16–18 (1948)

\bibitem{Ol-Sha-Yos} O.A. Oleinik, A.S. Shamaev, and G.A. Yosifian: {\it Mathematical problems in elasticity and homogenization}, Studies in
Mathematics and its Applications, {\bf26}. North-Holland Publishing Co., Amsterdam, 1992


\bibitem{Sapio} F. Sapio: {\it A dynamic model for viscoelasticity in domains with time dependent cracks}, NoDEA Nonlinear Differential Equations Appl. {\bf 28} (2021), no. 6, Paper No. 67, 47 pp.


\bibitem{Slepyan} L.I. Slepyan: { \it Models and phenomena in fracture mechanics}, Foundations of Engineering Mechanics. Springer-Verlag, Berlin, 2002.

\bibitem{Tasso} E. Tasso, {\it Weak formulation of elastodynamics in domains with growing cracks}, Ann. Mat. Pura Appl. (4) {\bf 199} (2020), 1571–1595.

\bibitem{Volterra_1} V. Volterra: {\it Sur les equations integro-differentielles et leurs applications},
Acta Mathem. {\bf 35} (1912), 295-356.

\bibitem{Volterra_2}  V. Volterra: {\it Le\c cons sur les fonctions de lignes}, Gauthier-Villars, Paris, 1913.



\end {thebibliography}
\end{document}